\documentclass[sigconf]{acmart}
\renewcommand\footnotetextcopyrightpermission[1]{} 
\usepackage{booktabs} 

\usepackage{algorithm} 
\usepackage{algorithmic} 
\usepackage{mdwlist} 
\usepackage[mathscr]{euscript} 
\usepackage{amsbsy} 
\usepackage{amsfonts} 
\usepackage{subfig} 
\usepackage{array} 
\usepackage{ragged2e} 

\usepackage{graphicx}
\usepackage{amsmath} 
\usepackage{hyperref} 
\usepackage{multirow} 
\usepackage{tablefootnote} 
\usepackage{color}
\usepackage{xspace}

\setcopyright{none}

\newcommand{\T}[1]{\boldsymbol{\mathscr{#1}}}

\newcommand{\bit}{\begin{itemize*}}
\newcommand{\eit}{\end{itemize*}}

\newcommand{\hide}[1]{}

\newcolumntype{C}[1]{>{\centering\arraybackslash}p{#1}}
\newcolumntype{R}[1]{>{\RaggedLeft\arraybackslash}p{#1}}

\begin{document}
\title{CTD: Fast, Accurate, and Interpretable Method for Static and \\Dynamic Tensor Decompositions
}


\author{Jungwoo Lee}
\affiliation{%
  \institution{Seoul National University}
}
\email{muon9401@gmail.com}

\author{Dongjin Choi}
\affiliation{%
  \institution{Seoul National University}
}
\email{skywalker5@snu.ac.kr}

\author{Lee Sael}
\affiliation{%
  \institution{The State University of New York (SUNY) Korea}
}
\email{sael@sunykorea.ac.kr}


\setlength{\textfloatsep}{0.2cm}
\setlength{\abovecaptionskip}{0.2cm}

\begin{abstract}
How can we find patterns and anomalies in a tensor, or multi-dimensional array, in an efficient and directly interpretable way?
How can we do this in an online environment, where a new tensor arrives each time step?
Finding patterns and anomalies in a tensor is a crucial problem with many applications, including building safety monitoring, patient health monitoring, cyber security, terrorist detection, and fake user detection in social networks.
Standard PARAFAC and Tucker decomposition results are not directly interpretable.
Although a few sampling-based methods have previously been proposed towards better interpretability, they need to be made faster, more memory efficient, and more accurate.

In this paper, we propose CTD, a fast, accurate, and directly interpretable tensor decomposition method based on sampling.
CTD-S, the static version of CTD, provably guarantees a high accuracy that is 17$\sim$83$\times$ more accurate than that of the state-of-the-art method.
Also, CTD-S is made 5$\sim$86$\times$ faster, and 7$\sim$12$\times$ more memory-efficient than the state-of-the-art method by removing redundancy.
CTD-D, the dynamic version of CTD, is the first interpretable dynamic tensor decomposition method ever proposed.
Also, it is made 2$\sim$3$\times$ faster than already fast CTD-S by exploiting factors at previous time step and by reordering operations.
With CTD, we demonstrate how the results can be effectively interpreted in the online distributed denial of service (DDoS) attack detection. 
\end{abstract}

%
%


\maketitle
\section{Introduction}
\label{sec:intro}
Given a tensor, or multi-dimensional array, how can we find patterns and anomalies in an efficient and directly interpretable way?
How can we do this in an online environment, where a new tensor arrives each time step?
Many real-world data are multi-dimensional and can be modeled as sparse tensors.
Examples include network traffic data (source IP - destination IP - time), movie rating data (user - movie - time), IoT sensor data, and healthcare data.
Finding patterns and anomalies in those tensor data is a very important problem with many applications such as building safety monitoring \cite{khoa2015}, patient health monitoring \cite{prada2012, wang2015,cyganek2015,perros2015}, cyber security \cite{thuraisingham2006}, terrorist detection \cite{clifton2010,allanach2004,arulselvan2009}, and fake user detection in social networks \cite{cao2012, kontaxis2011}.
Tensor decomposition method, a widely-used tool in tensor analysis, has been used for this task.
However, the standard tensor decomposition methods such as PARAFAC \cite{harshman1970foundations} and Tucker \cite{tucker1966some} do not provide interpretability and are not applicable for real-time analysis in environments with high-velocity data.

Sampling-based tensor decomposition methods \cite{drineas2007randomized,caiafa2010generalizing, mahoney2008tensor} arose as an alternative to due to their direct interpretability.
The direct interpretability not only reduces time and effort involved in finding patterns and anomalies from the decomposed tensors but also provides clarity in interpreting the result.
A sampling-based decomposition method for sparse tensors is also memory-efficient since it preserves the sparsity of the original tensors on the sampled factor matrices.
However, existing sampling-based tensor decomposition methods are slow, have high memory usage, and produce low accuracy.
For example, tensor-CUR \cite{mahoney2008tensor}, the state-of-the-art sampling-based static tensor decomposition method, has many redundant fibers including duplicates in its factors.
These redundant fibers cause higher memory usage and longer running time.
Tensor-CUR is also not accurate enough for real-world tensor analysis.

In addition to interpretability, demands for online method applicable in a dynamic environment, where multi-dimensional data are generated continuously at a fast rate, are also increasing.
A real-time analysis is not feasible with static methods since all the data, i.e., historical and incoming tensors, need to be decomposed over again at each time step.
There are a few dynamic tensor decomposition methods proposed \cite{sun2006beyond, sun2006window, DBLP:conf/kdd/ZhouVBJD16}. However, proposed methods are not directly interpretable and do not preserve sparsity.
To the best of our knowledge, there has been no sampling-based dynamic tensor decomposition method proposed.

\begin{table}[tbp]
	\small
	\setlength{\tabcolsep}{1pt}
	\caption{Comparison of our proposed CTD and the existing tensor-CUR. The static method CTD-S outperforms the state-of-the-art tensor-CUR in terms of time, memory usage, and accuracy. The dynamic method CTD-D is the fastest and the most accurate.} \label{tab:properties of the existing methods}
	\begin{center}
		{
			\begin{tabular}{C{2cm} C{2cm} C{2cm} C{2cm}}
				\toprule
				& Existing & \multicolumn{2}{c}{[Proposed]}  \\
				\cmidrule{2-4}
				& Tensor-CUR \cite{mahoney2008tensor} & \textbf{CTD-S} & \textbf{CTD-D}  \\
				
				\midrule
				
				\textbf{Interpretability}& \checkmark & \checkmark & \checkmark \\
				
				\textbf{Time} & fast & faster & \textbf{fastest} \\
				
				\textbf{Memory usage} & low & \textbf{lower} & low \\
				
				\textbf{Accuracy} & low & high & \textbf{highest} \\
				
				\textbf{Online} &  &  & \checkmark \\
				
				\bottomrule
				
			\end{tabular}
		}
	\end{center}
	\label{tab:comparison}
\end{table}

\begin{figure*} [t]
\subfloat[\textbf{Hypertext 2009}]
{	\includegraphics[width=0.23 \textwidth]{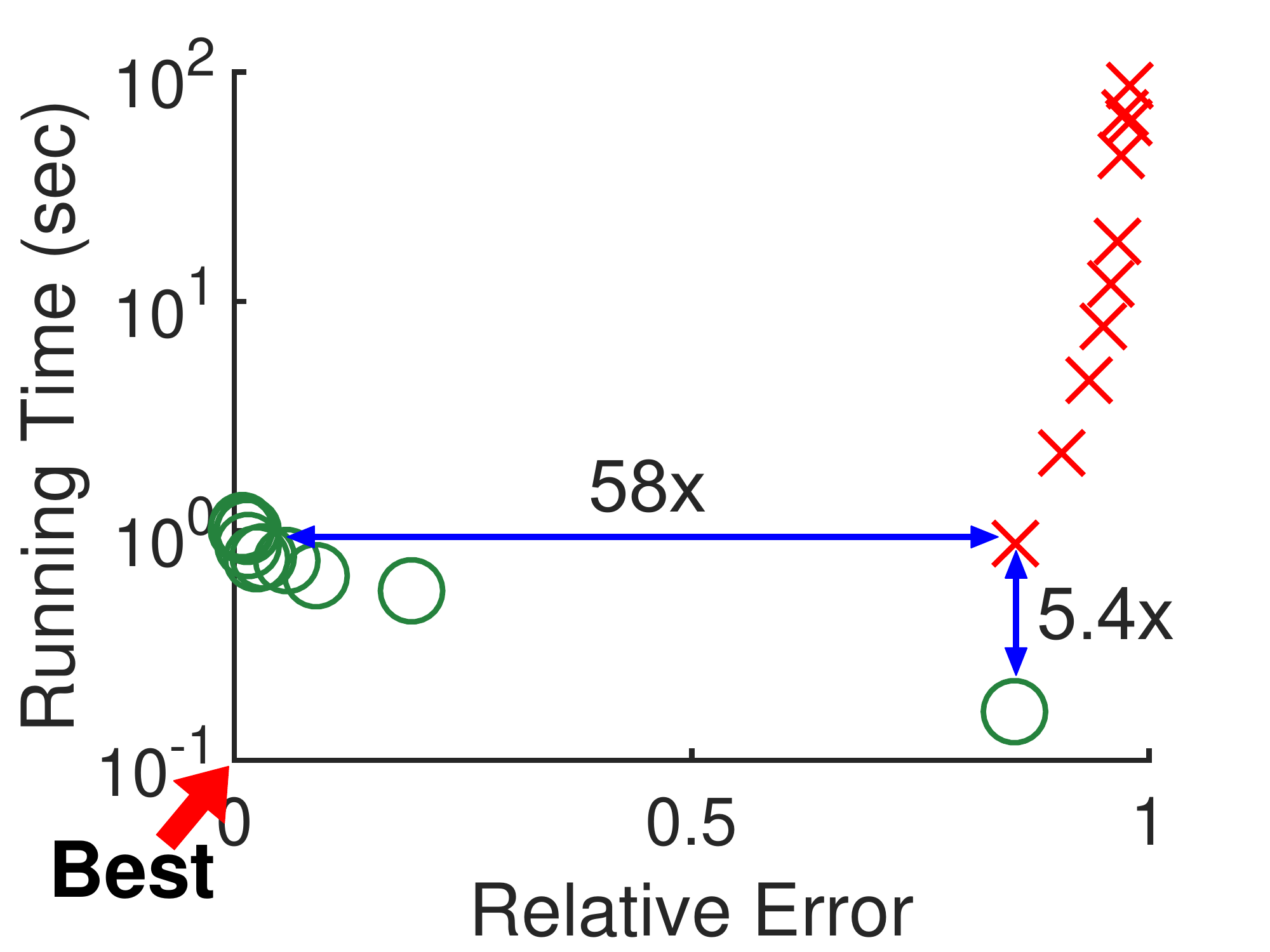}
}
\subfloat[\textbf{Haggle}]
{	\includegraphics[width=0.23 \textwidth]{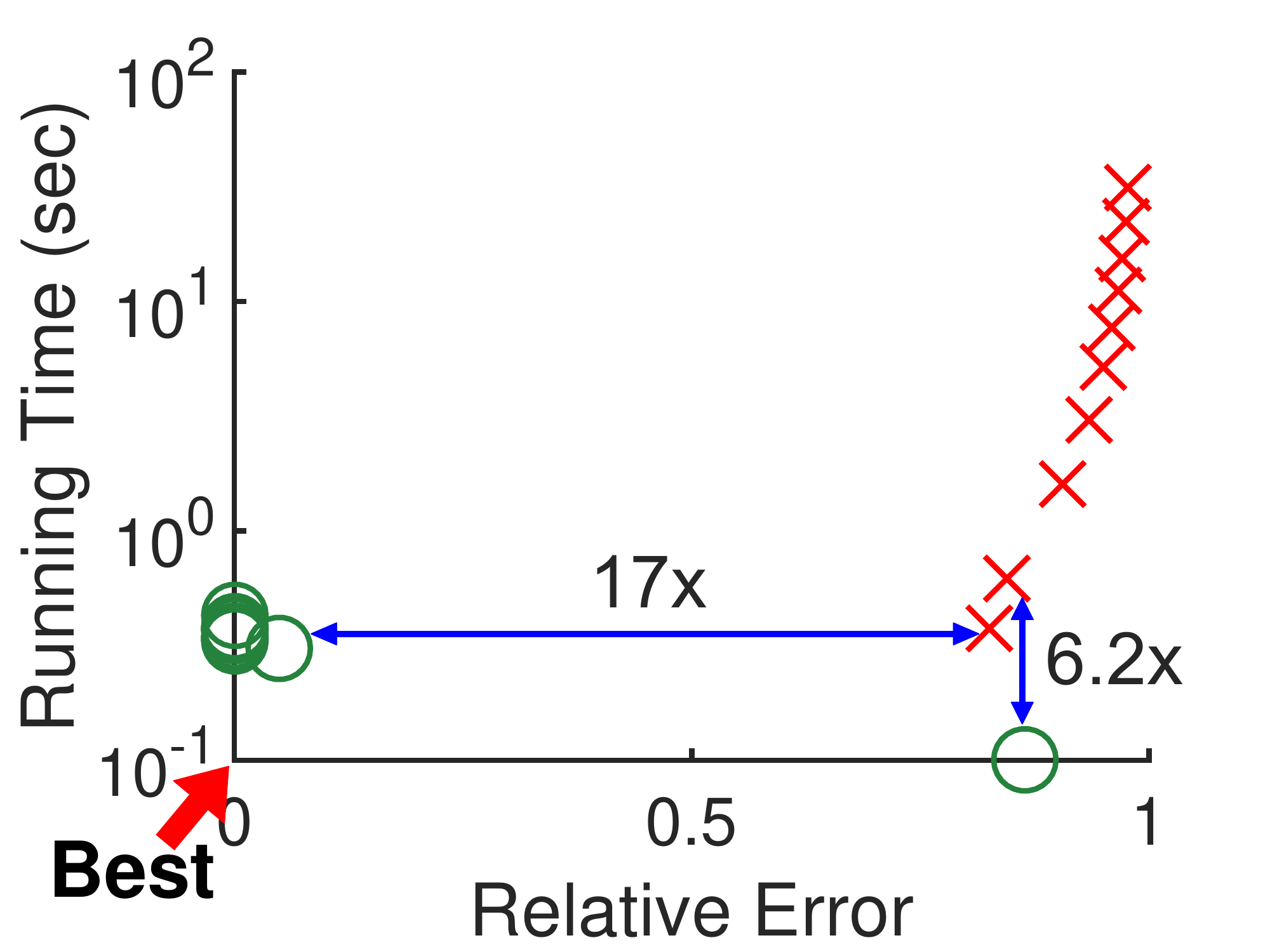}
}
\subfloat[\textbf{Manufacturing emails}]
{	\includegraphics[width=0.23 \textwidth]{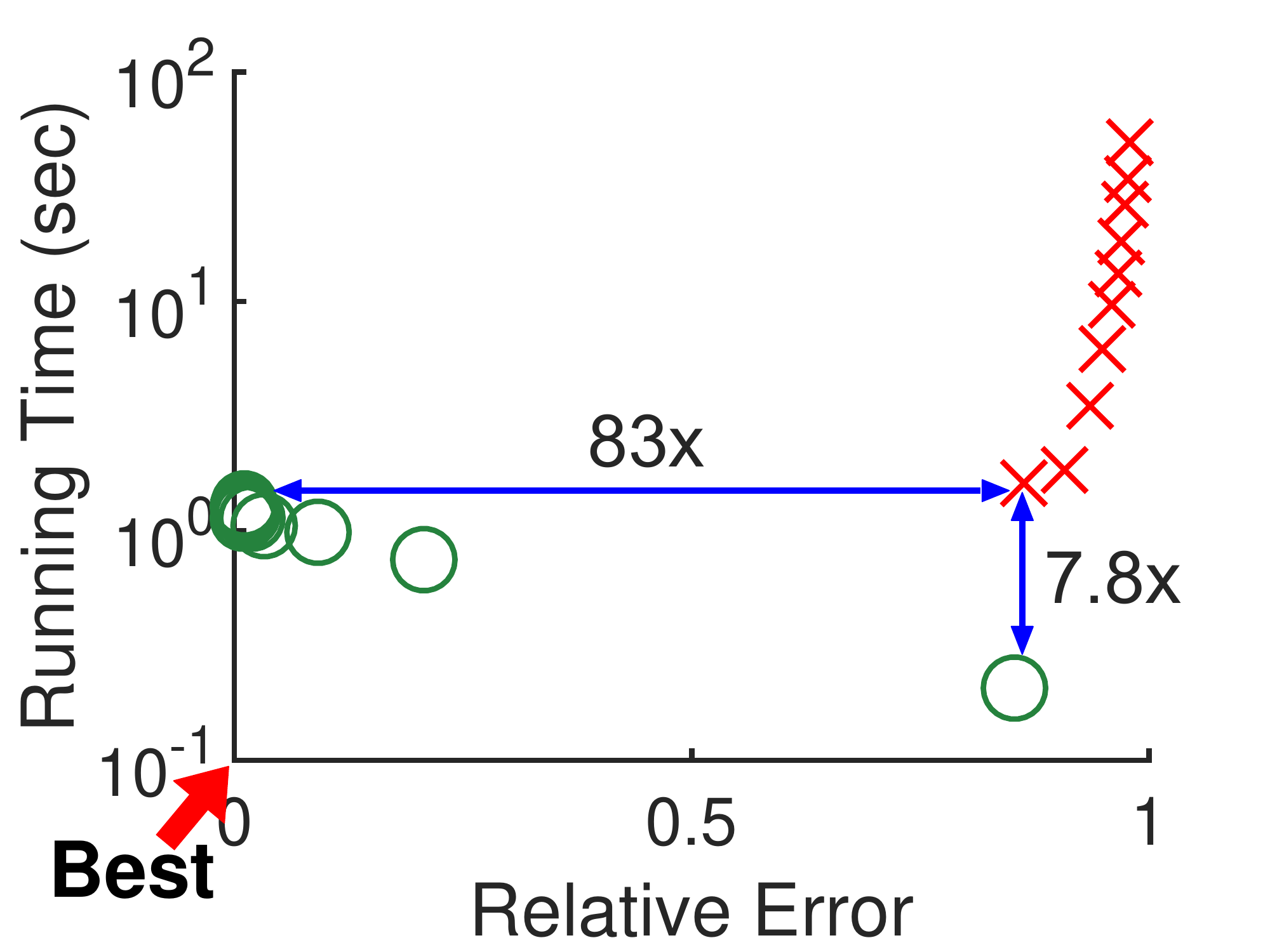}
}
\subfloat[\textbf{Infectious}]
{	\includegraphics[width=0.23 \textwidth]{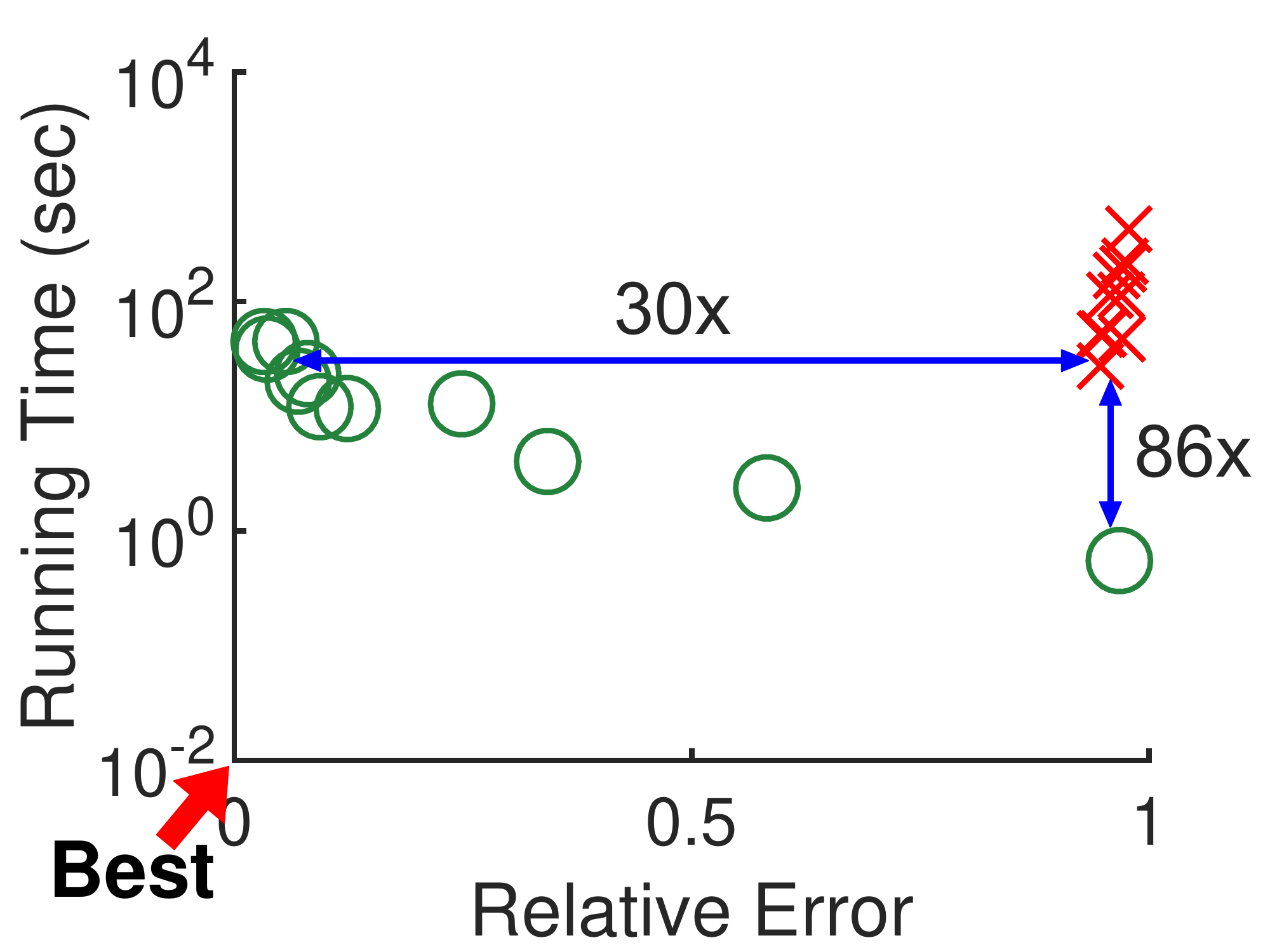}
}
\vspace*{-0.4cm}\\
\subfloat[\textbf{Hypertext 2009}]
{	\includegraphics[width=0.23 \textwidth]{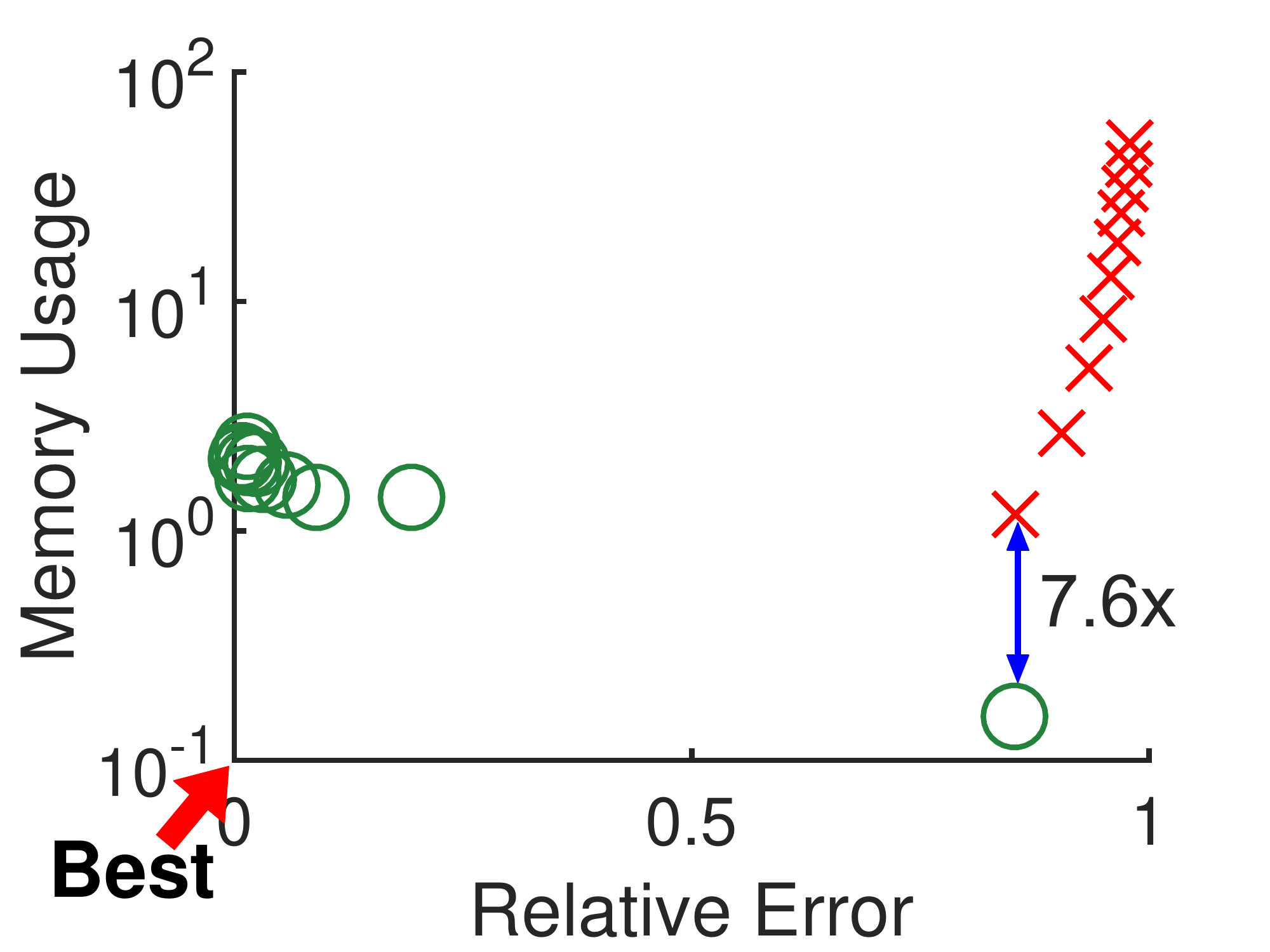}
}
\subfloat[\textbf{Haggle}]
{	\includegraphics[width=0.23 \textwidth]{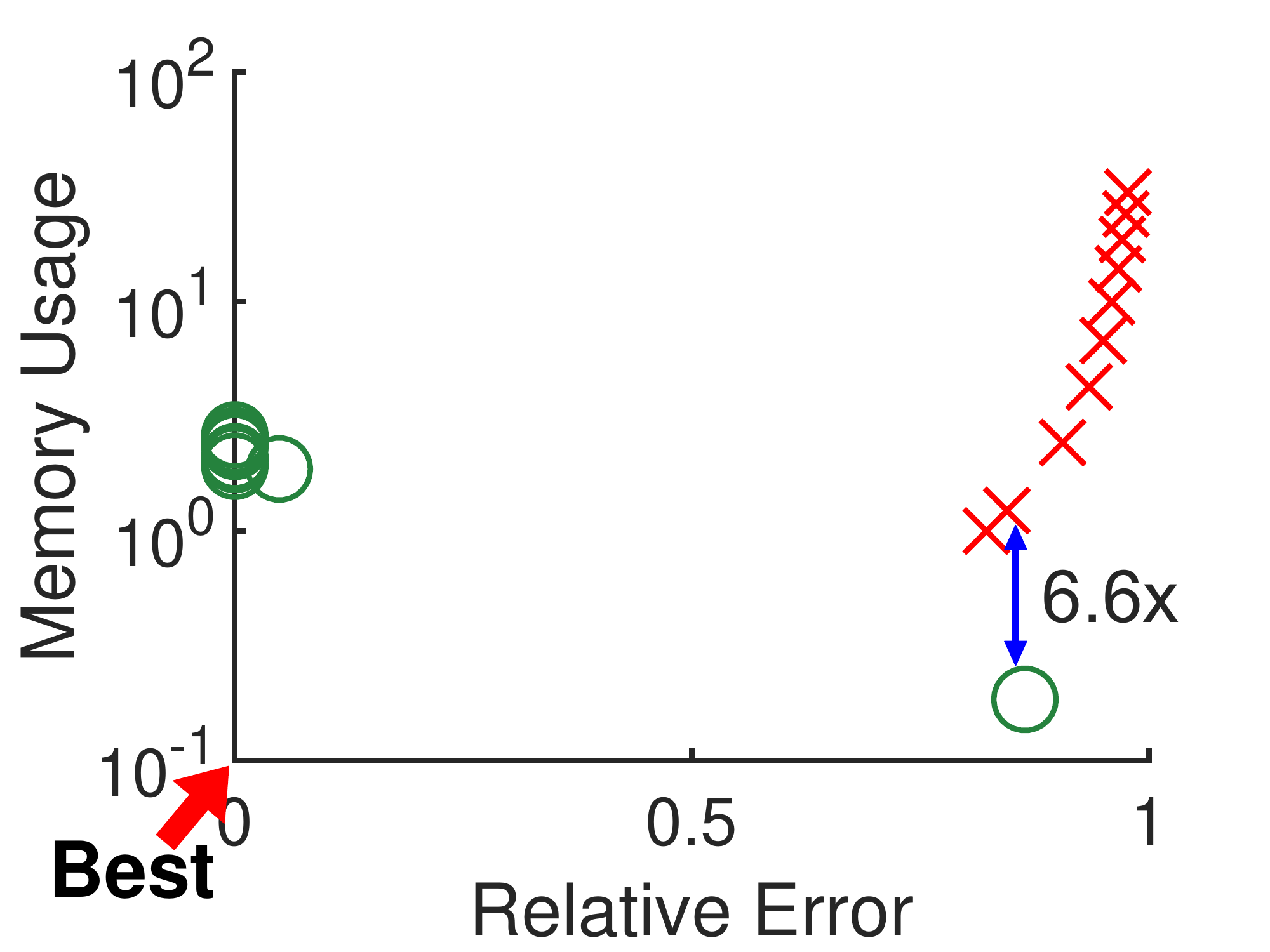}
}
\subfloat[\textbf{Manufacturing emails}]
{	\includegraphics[width=0.23 \textwidth]{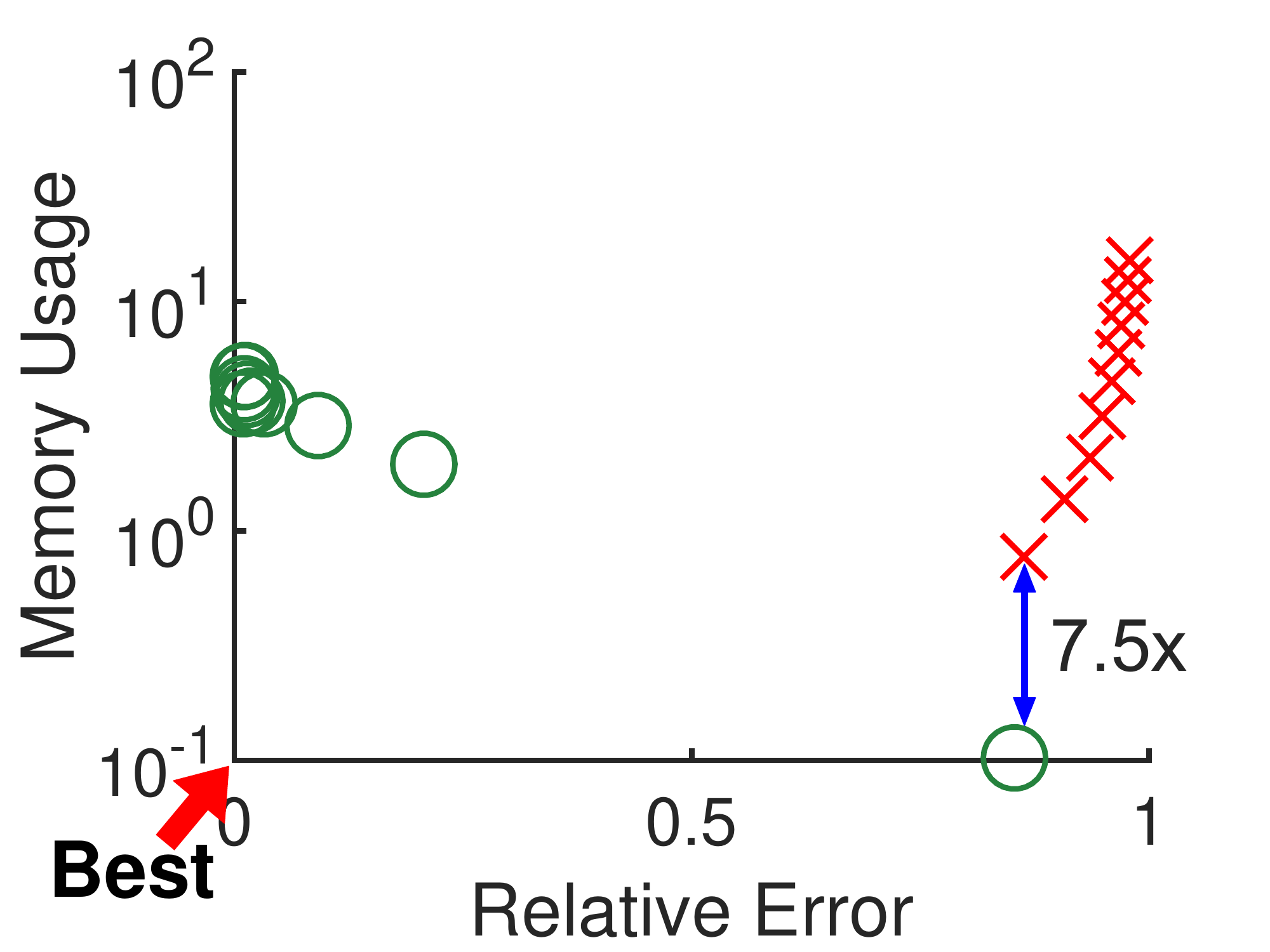}
}
\subfloat[\textbf{Infectious}]
{	\includegraphics[width=0.23 \textwidth]{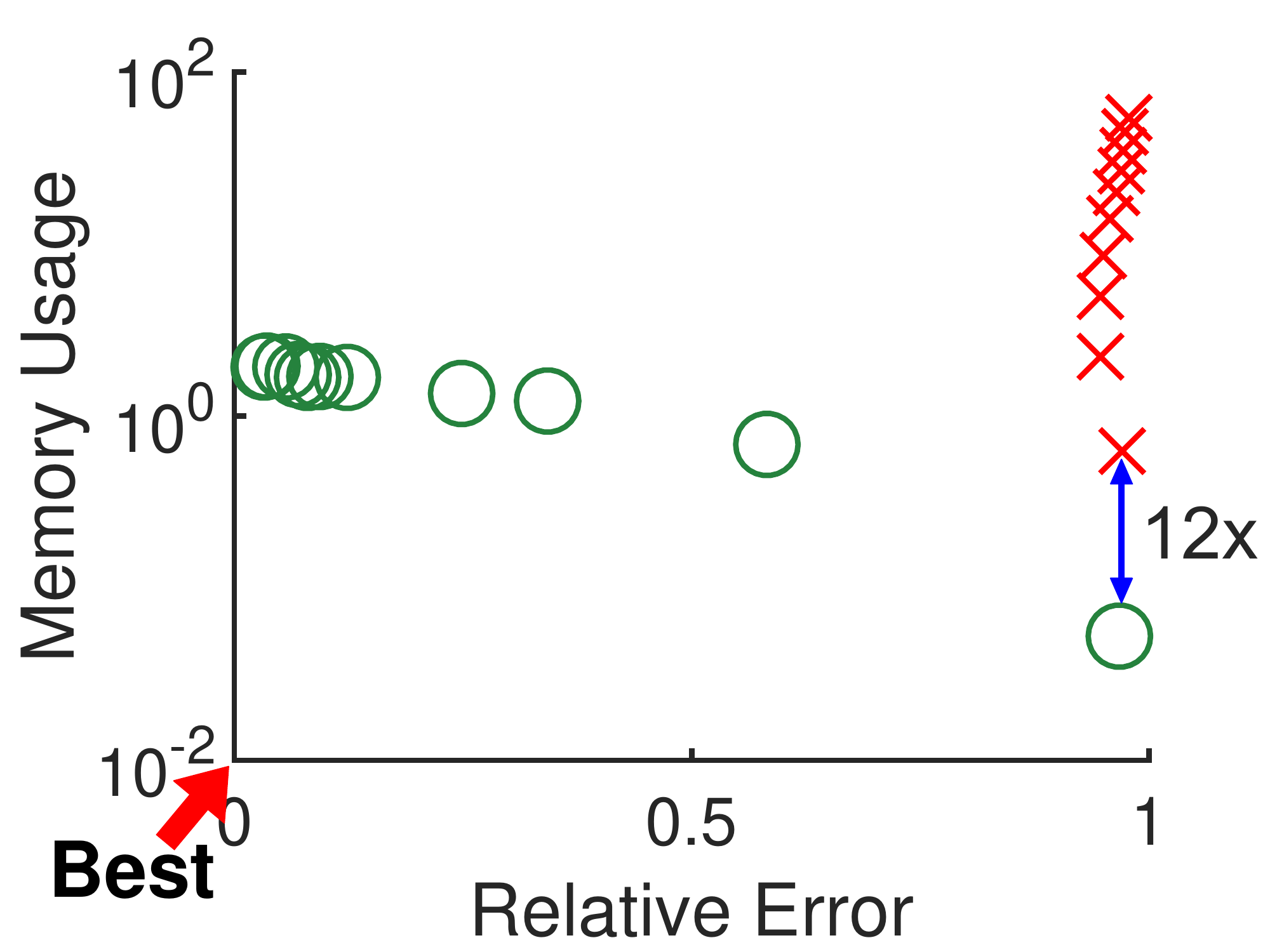}
}

\includegraphics[width=0.35 \textwidth]{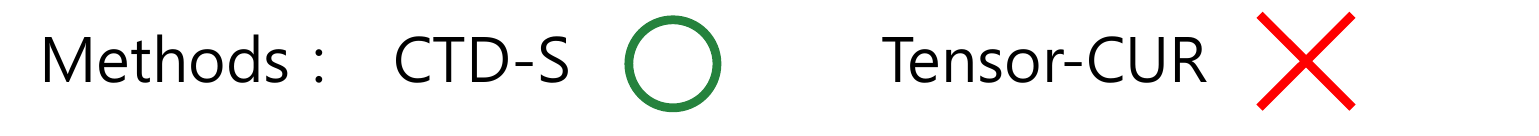}
\caption{Error, running time, and memory usage of CTD-S compared to those of tensor-CUR.  CTD-S is more accurate, faster and more memory-efficient than tensor-CUR.}
\label{fig:CTD-S all}
\end{figure*}

In this paper, we propose CTD (Compact Tensor Decomposition), a fast, accurate, and interpretable sampling-based tensor decomposition method.
CTD has two versions: CTD-S for static tensors, and CTD-D for dynamic tensors.
CTD-S is optimal after sampling, and results in a compact tensor decomposition through careful sampling and redundancy elimination, thereby providing much better running time and memory efficiency than previous methods.
CTD-D, the first sampling-based dynamic tensor decomposition method in literature, updates and modifies minimally on the components altered by the incoming data, making the method applicable for real-time analysis on a dynamic environment.
Table \ref{tab:comparison} shows the comparison of CTD and the existing method, tensor-CUR.
Our main contributions are as follows:
\bit
\item \textbf{Method.} We propose CTD, a fast, accurate, and directly interpretable tensor decomposition method. We prove the optimality of the static method CTD-S which makes it more accurate than the state-of-the-art method. Also, to the best of our knowledge, the dynamic method CTD-D is the first sampling-based dynamic tensor decomposition method.
\item \textbf{Performance.} CTD-S is 17$\sim$83$\times$ more accurate, 5$\sim$86$\times$ faster, and 7$\sim$12$\times$ more memory-efficient compared to tensor-CUR, the state-of-the-art competitor as shown in Figure \ref{fig:CTD-S all}.
CTD-D is up to $3 \times$ faster than CTD-S as shown in Figure \ref{fig:CTD-D all}.
\item \textbf{Interpretable Analysis.} We show how CTD results are directly interpreted to successfully detect DDoS attacks in network traffic data.
\eit

The code for CTD and datasets are available at \url{https://datalab.snu.ac.kr/CTD/}. The rest of this paper is organized as follows. Section \ref{sec:prelim} describes preliminaries and related works for tensor and sampling-based decomposition. Section \ref{sec:proposed} describes our proposed method CTD. Section \ref{sec:experiments} presents the experimental results. After presenting CTD at work in Section \ref{sec:ctd at works},
we conclude in Section \ref{sec:conclusions}. 

\section{Preliminaries and Related Works}
\label{sec:prelim}
In this section, we describe preliminaries and related works for tensor and sampling-based decompositions. Table \ref{tab:symbol_table} lists the definitions of symbols used in this paper.

\subsection{Tensor}
A tensor is a multi-dimensional array and is denoted by the boldface Euler script, e.g. $\T{X} \in \mathbb{R}^{I_1 \times \cdots \times I_N}$ where $N$ denotes the order (the number of axes) of $\T{X}$. Each axis of a tensor is also known as mode or way.
A fiber is a vector (1-mode tensor) which has fixed indices except one. Every index of a mode-$n$ fiber is fixed except $n$-th index. A fiber can be regarded as a higher-order version of a matrix row and column. A matrix column and row each correspond to mode-$1$ fiber and mode-$2$ fiber, respectively.
A slab is an $(N - 1)$-mode tensor which has only one fixed index.
$\mathbf{X}_{(\alpha)} \in \mathbb{R}^{I_\alpha \times N_\alpha}$ denotes a mode-$\alpha$ matricization of $\T{X}$, where $N_\alpha = \prod_{n \neq \alpha} I_n$. $\mathbf{X}_{(\alpha)}$ is made by rearranging mode-$\alpha$ fibers of $\T{X}$ to be the columns of $\mathbf{X}_{(\alpha)}$. $\Vert\T{X} \Vert_F$ is the Frobenius norm of $\T{X}$ and is defined by Equation \ref{eq:Frobenius norm of tensor}.
\begin{equation} \label{eq:Frobenius norm of tensor}
\Vert\T{X} \Vert_F^2 = \sum_{i_1, i_2, \cdots , i_N} x_{i_1 i_2 \cdots i_N}^2
\end{equation}
$\T{X} \times_n \mathbf{U} \in \mathbb{R}^{I_1 \times \cdots \times I_{n-1} \times J \times I_{n+1} \times \cdots \times I_N}$ denotes the $n$-mode product of a tensor $\T{X} \in \mathbb{R}^{I_1 \times \cdots \times I_N}$ with a matrix $\mathbf{U} \in \mathbb{R}^{J \times 	I_n}$. Elementwise,
\begin{equation} \label{eq:n_mode_product_elementwise}
(\T{X} \times_n \mathbf{U})_{i_1 \cdots i_{n-1} j i_{n+1} \cdots i_N} = \sum_{i_n = 1}^{I_n} x_{i_1 \cdots i_{n-1} i_n i_{n+1} \cdots i_N} u_{j i_n}
\end{equation}
$\T{X} \times_n \mathbf{U}$ has a property shown in Equation \ref{eq:n_mode_product_property}.
\begin{equation}\label{eq:n_mode_product_property}
	\T{Y} = \T{X} \times_n \mathbf{U} \Leftrightarrow \mathbf{Y}_{(n)} = \mathbf{U} \mathbf{X}_{(n)}
\end{equation}

We assume that a matrix or tensor is stored in a sparse-unordered representation (i.e. only nonzero entries are stored in a form of pair of indices and the corresponding value). $nnz(\T{X})$ denotes the number of nonzero elements in $\T{X}$.

We describe existing sampling-based matrix and tensor decomposition methods in the following subsections.

\begin{table}[t]
	\small
	\caption{Table of symbols.}
	\begin{tabular}{cl}
		\toprule
		\textbf{Symbol} & \textbf{Definition} \\
		\midrule
		$\T{X}$ & tensor (Euler script, bold letter) \\
		$\mathbf{X}$ & matrix (uppercase, bold letter) \\
		$\mathbf{x}$ & column vector (lower case, bold letter)\\
		$x$ & scalar (lower case, italic letter)\\
		$\mathbf{X}_{(n)}$ & mode-$n$ matricization of a tensor $\T{X}$\\
		$\mathbf{X}^\dagger$ & pseudoinverse of $\mathbf{X}$\\
		$N$ & order (number of modes) of a tensor\\
		$\times_n$ & $n$-mode product \\
		$\Vert \bullet \Vert_F$ & Frobenius norm\\
		$nnz(\T{X})$ & number of nonzero elements in $\T{X}$ \\
		\bottomrule
	\end{tabular}
	
	\label{tab:symbol_table}
\end{table}

\subsection{Sampling Based Matrix Decomposition}
Sampling-based matrix decomposition methods sample columns or rows from a given matrix and use them to make their factors.
They produce directly interpretable factors which preserve sparsity since those factors directly reflect the sparsity of the original data.
In contrast, a singular value decomposition (SVD) generates factors which are hard to understand and dense because the factors are in a form of linear combination of columns or rows from the given matrix.
Definition \ref{def:CX matrix decomposition} shows the definition for CX matrix decomposition \cite{drineas2008relative}, a kind of sampling-based matrix decomposition.
\begin{definition}\label{def:CX matrix decomposition}
	Given a matrix $\mathbf{A} \in \mathbb{R}^{m \times n}$, the matrix $\tilde{\mathbf{A}} = \mathbf{C} \mathbf{X}$ is a CX matrix decomposition of $\mathbf{A}$, where a matrix $\mathbf{C} \in \mathbb{R}^{m \times c}$ consists of actual columns of $\mathbf{A}$ and a matrix $\mathbf{X}$ is any matrix of size $c \times n$.
\end{definition}
We introduce well-known CX matrix decomposition methods: LinearTimeCUR, CMD, and Colibri.

\paragraph{LinearTimeCUR and CMD}
Drineas et al. \cite{drineas2006fast} proposed LinearTimeCUR and Sun et al. \cite{sun2007less} proposed CMD.
In the initial step, LinearTimeCUR and CMD sample columns from an original matrix $\mathbf{A}$ according to the probabilities proportional to the norm of each column with replacement. Drineas et al. \cite{drineas2006fast} has proven that this biased sampling provides an optimal approximation.
Then, they project $\mathbf{A}$ into the column space spanned by those sampled columns and use the projection as the low-rank approximation of $\mathbf{A}$.
LinearTimeCUR has many duplicates in its factors because a column or row with a higher norm is likely to be selected multiple times. These duplicates make LinearTimeCUR slow and require a large amount of memory.
CMD handles the duplication issue by removing duplicate columns and rows in the factors of LinearTimeCUR, thereby reducing running time and memory significantly.

\paragraph{Colibri}
Tong et al. \cite{tong2008colibri} proposed Colibri-S which improves CMD by removing all types of linear dependencies including duplicates.
Colibri-S is much faster and memory-efficient compared to LinearTimeCUR and CMD because the dimension of factors is much smaller than that of LinearTimeCUR and CMD.
Tong et al. \cite{tong2008colibri} also proposed the dynamic version Colibri-D. Although Colibri-D can update its factors incrementally, it fixes the indices of the initially sampled columns which need to be updated over time. Our CTD-D not only handles general dynamic tensors but also does not have to fix those indices.

\subsection{Sampling Based Tensor Decomposition}
Sampling-based tensor decomposition method samples actual fibers or slabs from an original tensor. In contrast to PARAFAC and Tucker, the most famous tensor decomposition methods, the resulting factors of sampling-based tensor decomposition method are easy to understand and usually sparse.
There are two types of sampling based tensor decomposition: one based on Tucker and the other based on LR tensor decomposition which is defined in Definition 2.2. 
In Tucker-type sampling based tensor decomposition (e.g., ApproxTensorSVD \cite{drineas2007randomized} and FBTD (fiber-based tensor decomposition)~\cite{caiafa2010generalizing}),
factor matrices for all modes are either sampled or generated; the overhead of generating a factor matrix for each mode makes these methods too slow for applications to real-time analysis. 
We focus on sampling methods based on LR tensor decomposition which is faster than those based on Tucker decomposition. 

\begin{definition}\label{def:LR tensor decomposition}
(LR tensor decomposition)
	Given a tensor $\T{X} \in \mathbb{R}^{\mathit{I}_1 \times \mathit{I}_2 \times \cdots \times \mathit{I}_N}$,
$\tilde{\T{X}} = \T{L} \times_\alpha \mathbf{R}$ is a mode-$\alpha$ LR tensor decomposition of $\T{X}$, where a matrix $\mathbf{R} \in \mathbb{R}^{I_\alpha \times c}$ consists of actual mode-$\alpha$ fibers of $\T{X}$ and a tensor $\T{L}$ is any tensor of size $I_1 \times \cdots \times I_{\alpha - 1} \times  c \times I_{\alpha + 1} \times \cdots \times I_N$.
\end{definition}


\paragraph{Tensor-CUR}
Mahoney et al. \cite{mahoney2008tensor} proposed tensor-CUR, a mode-$\alpha$ LR tensor decomposition method. Tensor-CUR is an $n$-dimensional extension of LinearTimeCUR. Tensor-CUR samples fibers and slabs from an original tensor and builds its factors using the sampled ones. The only difference between LinearTimeCUR and tensor-CUR is that tensor-CUR exploits fibers and slabs instead of columns and rows. Thus, tensor-CUR has drawbacks similar to those of LinearTimeCUR. Tensor-CUR has many redundant fibers in its factors and these fibers make tensor-CUR slow and use a large amount of memory.

\section{Proposed Method}
\label{sec:proposed}
In this section, we describe our proposed CTD (Compact Tensor Decomposition), an efficient and interpretable sampling-based tensor decomposition method.
We first describe the static version CTD-S, and then the dynamic version CTD-D of CTD.
\subsection{CTD-S for Static Tensors}

\paragraph{Overview.}
How can we design an efficient sampling-based static tensor decomposition method?
Tensor-CUR, the existing state-of-the-art, has many redundant fibers in its factors and these fibers make tensor-CUR slow and use large memory.
Our proposed CTD-S method removes all dependencies from the sampled fibers
and maintains only independent fibers; thus, CTD-S is faster and more memory-efficient than tensor-CUR.

\begin{figure} [h]
	\begin{center}
		\includegraphics[width=0.45 \textwidth]{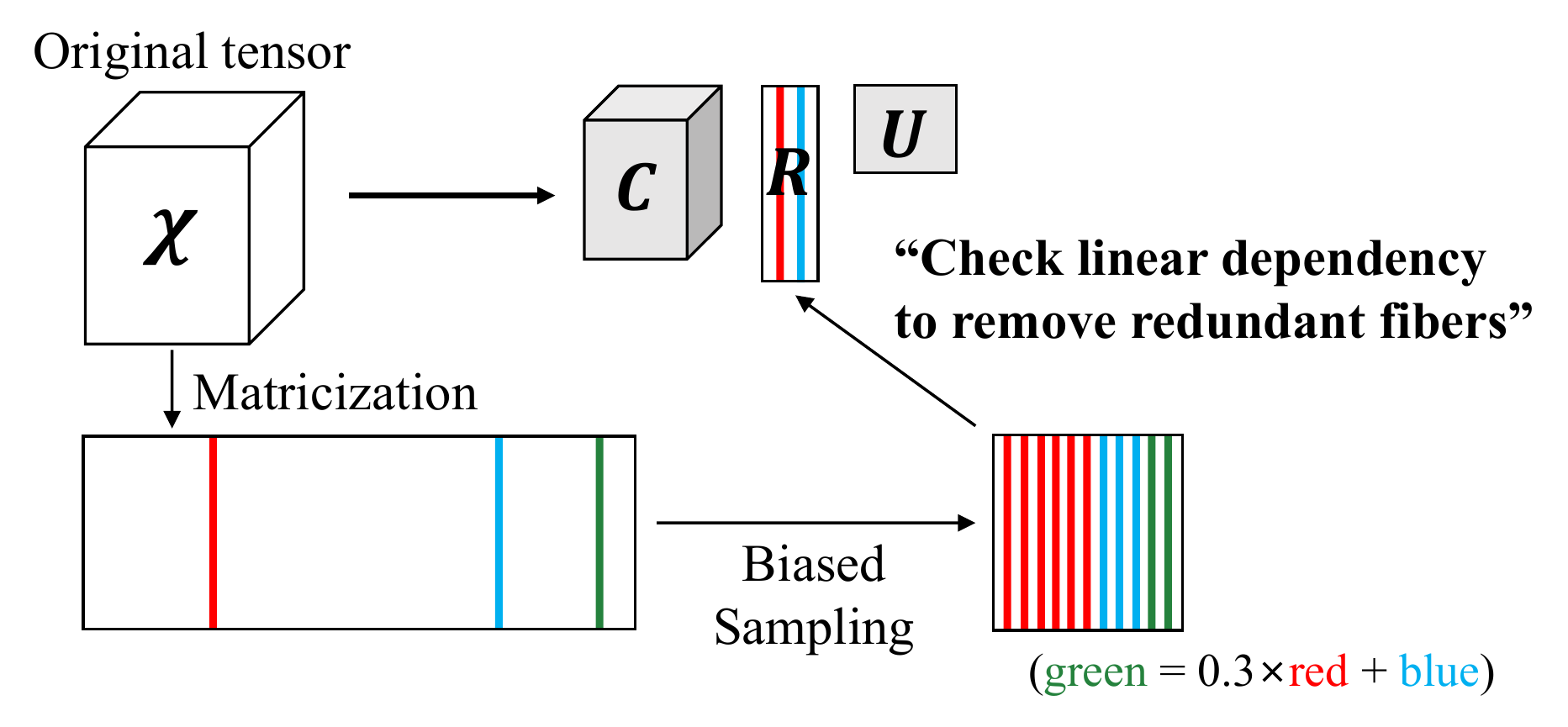}
	\end{center}
	\caption{The scheme for CTD-S.}
	\label{fig:CTD-S-overview}
\end{figure}
\paragraph{Algorithm.}
Figure \ref{fig:CTD-S-overview} shows the scheme for CTD-S. CTD-S first samples fibers biased toward a norm of each fiber. Three different fibers (red, blue, green) are sampled in Figure \ref{fig:CTD-S-overview}. There are many duplicates after biased sampling process since CTD-S samples fibers multiple times with replacement and a fiber with a higher norm is likely to be sampled many times. There also exist linearly dependent fibers such as the green fiber which can be expressed as a linear combination of the red one and the blue one. Those linearly dependent fibers including duplicates are redundant in that they do not give new information when interpreting the result. CTD-S removes those redundant fibers and stores only the independent fibers in its factor $\mathbf{R}$ to keep result compact. CTD-S only keeps one red fiber and one blue fiber in $\mathbf{R}$ in Figure \ref{fig:CTD-S-overview}.

CTD-S decomposes a tensor $\T{X} \in \mathbb{R}^{\mathit{I}_1 \times \mathit{I}_2 \times \cdots \times \mathit{I}_N}$ into one tensor $\T{C} \in \mathbb{R}^{\mathit{I}_1 \times \cdots \times \mathit{I}_{\alpha - 1} \times  \tilde{s} \times \mathit{I}_{\alpha + 1} \times \cdots \times \mathit{I}_N}$ , and two matrices $\mathbf{U} \in \mathbb{R}^{\tilde{s} \times \tilde{s}}$ and $\mathbf{R} \in \mathbb{R}^{\mathit{I}_{\alpha} \times \tilde{s}}$ such that $\T{X} \approx \T{C}\times_\alpha \mathbf{R}\mathbf{U}$. CTD-S is a mode-$\alpha$ LR tensor decomposition method and is interpretable since $\mathbf{R}$ consists of 	 independent fibers sampled from $\T{X}$.

\begin{algorithm} [h]
	\small
	\caption{CTD-S for Static Tensor} \label{alg:CTD-S}
	\begin{algorithmic}[1]
		\REQUIRE Tensor $\T{X} \in \mathbb{R}^{\mathit{I}_1 \times \mathit{I}_2 \times \cdots \times \mathit{I}_N}$, mode $\alpha \in \{1, \cdots ,N\}$, sample size $s \in \{1, \cdots ,N_{\alpha}\}$, and tolerance $\epsilon$
		
		\ENSURE $\T{C} \in \mathbb{R}^{\mathit{I}_1 \times \cdots \times \mathit{I}_{\alpha - 1} \times \tilde{s} \times \mathit{I}_{\alpha + 1} \times \cdots \times \mathit{I}_N}$, $\mathbf{U} \in \mathbb{R}^{\tilde{s} \times \tilde{s}}$, $\mathbf{R} \in \mathbb{R}^{\mathit{I}_{\alpha} \times \tilde{s}}$
		
		\STATE Let $\mathbf{X}_{(\alpha)}$ be the mode-$\alpha$ matricization of $\T{X}$
		
		\STATE Compute column distribution for $i = 1, \cdots , N_\alpha$: \\
		$P(i) \leftarrow \frac{ \vert \mathbf{X}_{(\alpha)}(:,i) \vert^2 }{\Vert \mathbf{X}_{(\alpha)} \Vert_F^2}$
		
		\STATE Sample $s$ columns from $\mathbf{X}_{(\alpha)}$ based on $P(i)$. Let $I = \{i_1, \cdots , i_s\}$
		
		\STATE Let $I' = \{i'_1, \cdots , i'_{s'}\}$ be a set consisting of unique elements in $I$
		
		\STATE Initialize $\mathbf{R} \leftarrow [\mathbf{X}_{(\alpha)}(:,i'_1)]$ and $\mathbf{U} \leftarrow 1/(\mathbf{X}_{(\alpha)}(:,i'_1)^T\mathbf{X}_{(\alpha)}(:,i'_1))$
		
		\FOR{$k = 2 : s'$}
		
		\STATE Compute the residual: \\
		$\overrightarrow{res} \leftarrow (\mathbf{X}_{(\alpha)}(:,i'_k) - \mathbf{R}\mathbf{U}\mathbf{R}^T\mathbf{X}_{(\alpha)}(:,i'_k))$
		
		\IF {$||\overrightarrow{res}|| \leq \epsilon ||\mathbf{X}_{(\alpha)}(:,i'_k)||$}
		
		\STATE continue
		
		\ELSE
		
		\STATE Compute: $\delta \leftarrow ||\overrightarrow{res}||^2$ and $\overrightarrow{\mathbf{y}} \leftarrow \mathbf{U}\mathbf{R}^T\mathbf{X}_{(\alpha)}(:, i'_k)$
		
		\STATE Update $\mathbf{U}$:\\
		$\mathbf{U} \leftarrow $
		$\begin{pmatrix}
		\mathbf{U} + \overrightarrow{\mathbf{y}} \overrightarrow{\mathbf{y}}^T / \delta & - \overrightarrow{\mathbf{y}} / \delta \\
		- \overrightarrow{\mathbf{y}}^T / \delta & 1 / \delta \\
		\end{pmatrix}$
		
		\STATE Expand $\mathbf{R}: \mathbf{R} \leftarrow [\mathbf{R}, \mathbf{X}_{(\alpha)}(:,i'_k)]$
		
		\ENDIF
		
		\ENDFOR
		
		\STATE Compute $\T{C} \leftarrow \T{X} \times_\alpha \mathbf{R}^T$
		
		\RETURN $\T{C}$, $\mathbf{U}$, $\mathbf{R}$
	\end{algorithmic}
\end{algorithm}
Algorithm \ref{alg:CTD-S} shows the procedure of CTD-S. First, CTD-S computes the probabilities of mode-$\alpha$ fibers of $\T{X}$, which are proportional to the norm of each fiber, and then samples $s$ fibers from $\T{X}$ according to the probabilities with replacement, in lines 1-3. Redundant fibers exist in the sampled fibers in this step. CTD-S selects unique fibers from the initially sampled $s$ fibers in line 4 where $s'$ denotes the number of those unique fibers. This step reduces the number of iterations in lines 6-15 from $s-1$ to $s'-1$. $\mathbf{R}$ is initialized by the first sampled fiber in line 5. In lines 6-15, CTD-S removes redundant mode-$\alpha$ fibers in the sampled fibers. The matrices $\mathbf{U}$ and $\mathbf{R}$ are computed incrementally in this step. The columns of $\mathbf{R}$ always consist of independent mode-$\alpha$ fibers through the loop. In each iteration, CTD-S checks whether one of the sampled fibers is linearly independent of the column space spanned by $\mathbf{R}$ or not in lines 7-8. If the fiber is independent, CTD-S updates $\mathbf{U}$ and expands  $\mathbf{R}$ with the fiber in lines 10-13. Finally, CTD-S computes $\T{C}$ with $\T{X}$ and $\mathbf{R}$ in line 16.

Lemma \ref{lemma: CTD-S complexity} shows the computational cost of CTD-S.
\begin{lemma} \label{lemma: CTD-S complexity}
	The computational complexity of CTD-S is  $\mathcal{O}((\tilde{s} I_\alpha + s) N_\alpha + s' (\tilde{s}^2 + nnz(\mathbf{R})) + s \log s + nnz(\T{X}))$, where $N_\alpha$ is $\prod_{n \neq \alpha} I_n$ and $\tilde{s} \ll s' \leq s$.
\end{lemma}
\begin{proof}
	 The mode-$\alpha$ matricization of $\T{X}$ in line 1 needs $\mathcal{O}(nnz(\T{X}))$ operations.
	 Computing column distribution in line 2 takes $\mathcal{O}(nnz(\T{X}) + N_\alpha)$ and sampling $s$ columns in line 3 takes $\mathcal{O}(sN_\alpha)$.
	 $\mathcal{O}(s \log s)$ operation is required in computing unique elements in $I$ in line 4.
	 Computing $\mathbf{R}$ and $\mathbf{U}$ in lines 5-15 takes $\mathcal{O}(s' (\tilde{s}^2 + nnz(\mathbf{R})))$ as proved in Lemma 1 in \cite{tong2008colibri}.
	 Computing $\T{C}$ in line 16 takes $\mathcal{O}(\tilde{s} I_\alpha N_\alpha)$.
	 Overall, CTD-S needs $\mathcal{O}((\tilde{s} I_\alpha + s) N_\alpha + s' (\tilde{s}^2 + nnz(\mathbf{R})) + s \log s + nnz(\T{X}))$ operations.\hfill
\end{proof}

Lemma \ref{lemma: CTD-S accuracy} shows that CTD-S has the optimal accuracy for given sampled fibers, and thus is more accurate than tensor-CUR. 

\begin{lemma} \label{lemma: CTD-S accuracy}
	CTD-S is more accurate than tensor-CUR. For a given $\mathbf{R}_0$ consisting of initially sampled fibers, CTD-S has the minimum error.
\end{lemma}
\begin{proof}
	CTD-S and tensor-CUR are both mode-$\alpha$ LR tensor decomposition methods and have errors in the form of Equation \ref{LR decomposition error}.
	\begin{equation} \label{LR decomposition error}
	||\T{X} - \T{L}\times_\alpha \mathbf{R}||_F = ||\mathbf{X}_{(\alpha)} - \mathbf{R} \mathbf{L}_{(\alpha)}||_F
	\end{equation}
	where the equality 
comes from Equation \ref{eq:n_mode_product_property}. They both sample fibers from $\T{X}$ in the same way in their initial step. Let $\mathbf{R}_0$ be the matrix consisting of those initially sampled fibers. Assume the same $\mathbf{R}_0$ is given for both methods. Then, the error is now a function of ${\mathbf{L}_{(\alpha)}}$ as shown in Equation \ref{objective function of LR decomposition}. Equation \ref{minimizes LR decomposition error} shows the ${\mathbf{L}_{(\alpha)}}$ which minimizes the error.
	\begin{equation} \label{objective function of LR decomposition}
	f(\mathbf{L}_{(\alpha)}) = ||\mathbf{X}_{(\alpha)} - \mathbf{R}_0 \mathbf{L}_{(\alpha)}||_F
	\end{equation}
	\begin{equation} \label{minimizes LR decomposition error}
	\underset{\mathbf{L}_{(\alpha)}} {\arg \min} f(\mathbf{L}_{(\alpha)}) = \mathbf{R}_0^\dagger \mathbf{X}_{(\alpha)}
	\end{equation}
	We show CTD-S has the minimum error. Let $\mathbf{R}$ consists of the independent columns of $\mathbf{R}_0$.
	\begin{equation} \label{eq: CTD-S has minimum error}
	\begin{aligned}
		||\mathbf{X}_{(\alpha)} - \mathbf{R}_0 \mathbf{R}_0^\dagger \mathbf{X}_{(\alpha)}||_F & = ||\mathbf{X}_{(\alpha)} - \mathbf{R} \mathbf{R}^\dagger \mathbf{X}_{(\alpha)}||_F \\
		& = ||\mathbf{X}_{(\alpha)} - \mathbf{R} (\mathbf{R}^T \mathbf{R})^{-1} \mathbf{R}^T \mathbf{X}_{(\alpha)}||_F \\
		& = ||\mathbf{X}_{(\alpha)} - \mathbf{R} \mathbf{U} \mathbf{C}_{(\alpha)}||_F \\
		& = \text{ Error of CTD-S}
	\end{aligned}
	\end{equation}
	The first equality in Equation \ref{eq: CTD-S has minimum error} holds because $\mathbf{R}_0 \mathbf{R}_0^\dagger \mathbf{X}_{(\alpha)}$ means the projection of $\mathbf{X}_{(\alpha)}$ onto the column space of $\mathbf{R}_0$, and $\mathbf{R}$ and $\mathbf{R}_0$ have the same column space.
	The third equality holds because CTD-S uses $(\mathbf{R}^T \mathbf{R})^{-1}$ for its factor $\mathbf{U}$ (theorem 1 in \cite{tong2008colibri}), and $\mathbf{R}^T \mathbf{X}_{(\alpha)}$ for its factor $\T{C}$.
	In contrast, tensor-CUR does not have the minimum error because $f(\mathbf{L}_{(\alpha)})$ in the Equation \ref{objective function of LR decomposition} is convex and tensor-CUR has $\mathbf{L}_{(\alpha)}$ which is different from $\mathbf{R}_0^\dagger \mathbf{X}_{(\alpha)}$. Specifically, tensor-CUR further samples rows (called slabs) from $\mathbf{X}_{(\alpha)}$ to construct its $\mathbf{L}_{(\alpha)}$. \hfill
\end{proof}

\subsection{CTD-D for Dynamic Tensors}
\paragraph{Overview.}
How can we design an efficient sampling-based dynamic tensor decomposition method?
In a dynamic setting, a new tensor arrives at every time step and we want to keep track of sampling-based tensor decomposition.
The main challenge is to update factors quickly while preserving accuracy.
Note that there has been no sampling-based dynamic tensor decomposition method in the literature.
Applying CTD-S at every time step is not a feasible option since it starts from scratch to update its factors, and thus running time increases rapidly as tensor grows.
We propose CTD-D, the first sampling-based dynamic tensor decomposition method.
CTD-D samples mode-$\alpha$ fibers only from the newly arrived tensor,
and then updates the factors appropriately using those sampled ones.
The main idea of CTD-D is to update the factors of CTD-S incrementally by (1) exploiting factors at previous time step and (2) reordering operations.
\begin{figure} [t]
	\begin{center}
		\includegraphics[width=0.45 \textwidth]{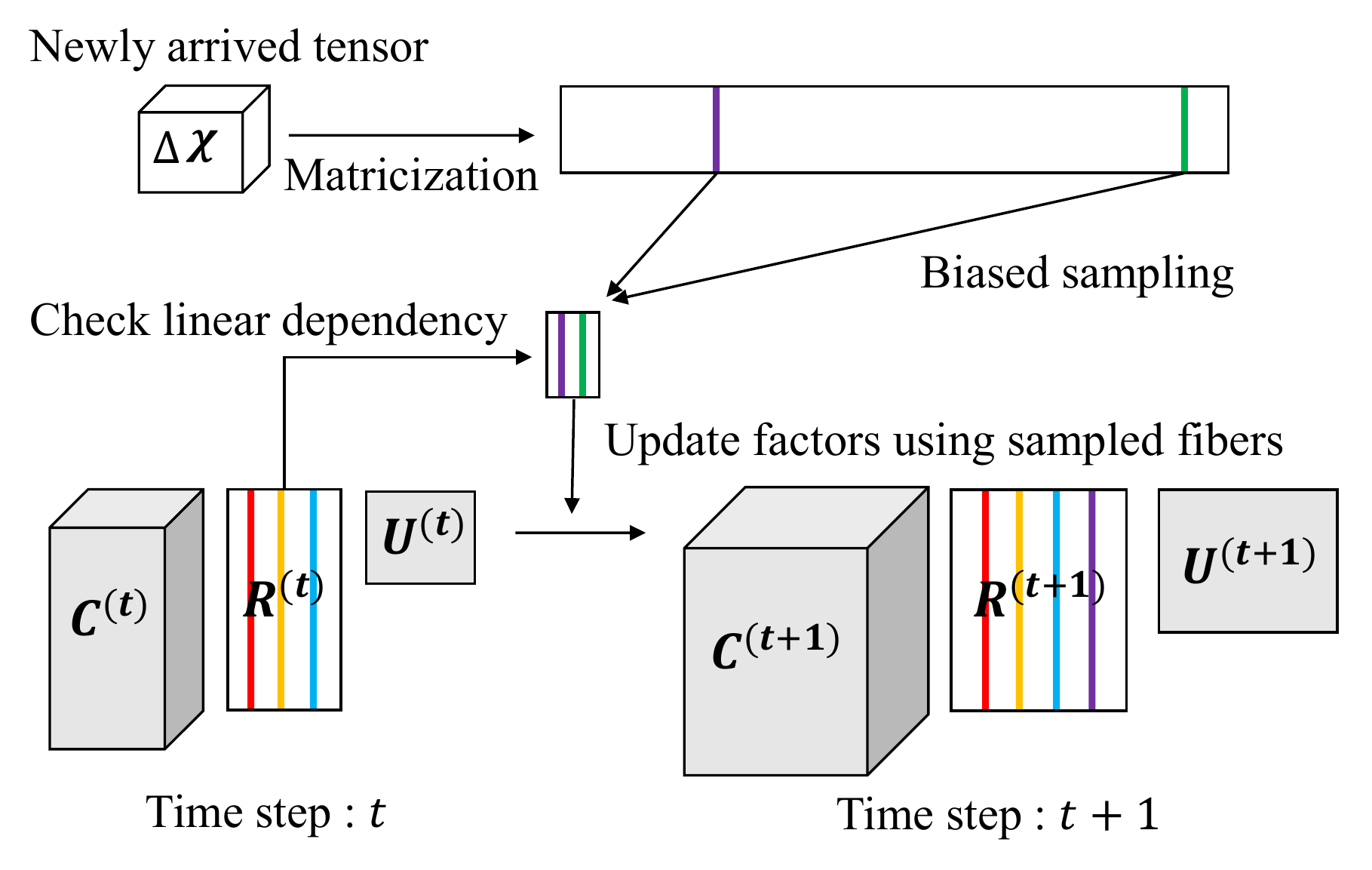}
	\end{center}
	\caption{The scheme for CTD-D.}
	\label{fig:CTD-D-overview}
\end{figure}
\begin{algorithm} [t]
	\small
	\caption{CTD-D for Dynamic Tensor} \label{alg:CTD-D}
	\begin{algorithmic}[1]
		\REQUIRE Tensor $\Delta \T{X} \in \mathbb{R}^{\mathit{I}_1 \times \cdots \times \mathit{I}_{N-1} \times 1}$, mode $\alpha \in \{1, \cdots ,N-1\}$, $\T{C}^{(t)}$, $\mathbf{U}^{(t)}$, $\mathbf{R}^{(t)}$, sample size $d \in \{1, \cdots ,\Delta N_\alpha\}$, and tolerance $\epsilon$
		
		\ENSURE $\T{C}^{(t+1)}$, $\mathbf{U}^{(t+1)}$, $\mathbf{R}^{(t+1)}$
		
		\STATE Let $\Delta \mathbf{X}_{(\alpha)}$ be the mode-$\alpha$ matricization of $\Delta \T{X}$
		
		\STATE Compute column distribution for $i = 1, \cdots, \Delta N_\alpha$: \\ $P(i) \leftarrow \frac{ \vert \Delta \mathbf{X}_{(\alpha)}(:,i) \vert^2 }{\Vert \Delta \mathbf{X}_{(\alpha)} \Vert_F^2}$
		
		\STATE Sample $d$ columns from $\Delta \mathbf{X}_{(\alpha)}$ based on $P(i)$. Let $I = \{i_1, \cdots , i_d\}$
		
		\STATE Let $I' = \{i'_1, \cdots , i'_{d'}\}$ be a set consisting of unique elements in $I$
		
		\STATE Initialize $\mathbf{R}^{(t+1)} \leftarrow \mathbf{R}^{(t)}$, $\mathbf{U}^{(t+1)} \leftarrow \mathbf{U}^{(t)}$, and $\Delta \mathbf{R} \leftarrow []$
		
		\FOR {$k = 1 : d'$}
		
		\STATE Let $\mathbf{x} \leftarrow \Delta \mathbf{X}_{(\alpha)}(:,i'_k)$
		
		\STATE Compute the residual: \\
		$\overrightarrow{res} \leftarrow (\mathbf{x} - \mathbf{R}^{(t+1)} \mathbf{U}^{(t+1)} (\mathbf{R}^{(t+1)})^T \mathbf{x})$
		
		\IF {$||\overrightarrow{res}|| \leq \epsilon ||\mathbf{x}||$}
		
		\STATE continue
		
		\ELSE
		
		\STATE Compute: $\delta \leftarrow ||\overrightarrow{res}||^2$ and $\overrightarrow{\mathbf{y}} \leftarrow \mathbf{U}^{(t+1)}(\mathbf{R}^{(t+1)})^T\mathbf{x}$
		
		\STATE Update $\mathbf{U}^{(t+1)}$: \\
		$\mathbf{U}^{(t+1)} \leftarrow $
		$\begin{pmatrix}
		\mathbf{U}^{(t+1)} + \overrightarrow{\mathbf{y}} \overrightarrow{\mathbf{y}}^T / \delta & - \overrightarrow{\mathbf{y}} / \delta \\
		- \overrightarrow{\mathbf{y}}^T / \delta & 1 / \delta \\
		\end{pmatrix}$
		
		\STATE Expand $\mathbf{R}^{(t+1)}$ and $\Delta \mathbf{R} : \mathbf{R}^{(t+1)} \leftarrow [\mathbf{R}^{(t+1)}, \mathbf{x}]$ and $\Delta \mathbf{R} \leftarrow [\Delta \mathbf{R}, \mathbf{x}]$
		
		\ENDIF
		
		\ENDFOR
		
		Update $\mathbf{C}_{(\alpha)}^{(t+1)}$ : \\
		
		\IF {$\Delta \mathbf{R}$ is not empty}
			
			\STATE $\mathbf{C}_{(\alpha)}^{(t+1)} \leftarrow$
			$\begin{pmatrix}
			\mathbf{C}_{(\alpha)}^{(t)} & (\mathbf{R}^{(t)})^T \Delta \mathbf{X}_{(\alpha)} \\
			(\Delta \mathbf{R})^T \mathbf{R}^{(t)} \mathbf{U}^{(t)} \mathbf{C}_{(\alpha)}^{(t)} & (\Delta \mathbf{R})^T \Delta \mathbf{X}_{(\alpha)} \\
			\end{pmatrix}$
			
			\ELSE
			
			\STATE $\mathbf{C}_{(\alpha)}^{(t+1)} \leftarrow$
			$\begin{pmatrix}
			\mathbf{C}_{(\alpha)}^{(t)} & (\mathbf{R}^{(t)})^T \Delta \mathbf{X}_{(\alpha)} \\
			\end{pmatrix}$
			
			\ENDIF
		
		\STATE Fold $\mathbf{C}_{(\alpha)}^{(t+1)}$ into $\T{C}^{(t+1)}$
		
		\RETURN $\T{C}^{(t+1)}$, $\mathbf{U}^{(t+1)}$, $\mathbf{R}^{(t+1)}$
	\end{algorithmic}
\end{algorithm}

\paragraph{Algorithm.}
Figure \ref{fig:CTD-D-overview} shows the scheme for CTD-D.
At each time step, CTD-D samples fibers from newly arrived tensor and updates factors by checking linear dependency of sampled fibers with the factor at previous time step.
Purple and green fiber are sampled from newly arrived tensor in Figure \ref{fig:CTD-D-overview}. Note that the purple fiber is added to the factor $\mathbf{R}$ since it is linearly independent of the fibers in the factor at the previous time step, while the linearly dependent green fiber is ignored.

For any time step $t$, CTD-D maintains its factors \\
$\T{C}^{(t)} \in \mathbb{R}^{\mathit{I}_1 \times \cdots \times \mathit{I}_{\alpha - 1} \times \tilde{d}_t \times \mathit{I}_{\alpha + 1} \times \cdots \times \mathit{I}_{N-1} \times t}$,
$\mathbf{U}^{(t)} \in \mathbb{R}^{\tilde{d}_t \times \tilde{d}_t}$, and $\mathbf{R}^{(t)} \in \mathbb{R}^{I_\alpha \times \tilde{d}_t}$ such that $\T{X}^{(t)} \approx \T{C}^{(t)}\times_\alpha \mathbf{R}^{(t)}\mathbf{U}^{(t)}$, where the upper subscript $(t)$ indicates that the factor is at time step $t$.
$\T{X}^{(t)}$ grows along the time mode and we assume that $N$-th mode is the time mode in a dynamic setting, where $N$ denotes the order of $\T{X}^{(t)}$.
At the next time step $t + 1$, CTD-D receives newly arrived tensor $\Delta\T{X} \in \mathbb{R}^{\mathit{I}_1 \times \mathit{I}_2 \times \cdots \times \cdots \times \mathit{I}_{N-1} \times 1}$ and updates $\T{C}^{(t)}$, $\mathbf{U}^{(t)}$, and $\mathbf{R}^{(t)}$ into $\T{C}^{(t+1)} \in \mathbb{R}^{\mathit{I}_1 \times \cdots \times \mathit{I}_{\alpha - 1} \times \tilde{d}_{t+1} \times \mathit{I}_{\alpha + 1} \times \cdots \times \mathit{I}_{N-1} \times (t + 1) }$, $\mathbf{U}^{(t+1)} \in \mathbb{R}^{\tilde{d}_{t+1} \times \tilde{d}_{t+1}}$, and $\mathbf{R}^{(t+1)} \in \mathbb{R}^{I_\alpha \times \tilde{d}_{t+1}}$, respectively such that $\T{X}^{(t+1)} \approx \T{C}^{(t+1)}\times_\alpha \mathbf{R}^{(t+1)}\mathbf{U}^{(t+1)}$.

Algorithm \ref{alg:CTD-D} shows the procedure of CTD-D.
First, CTD-D computes the probabilities of mode-$\alpha$ fibers of $\Delta \T{X}$, which are proportional to the norm of each fiber, and then samples $d$ fibers according to the probabilities with replacement in lines 1-3.
CTD-D selects unique $d'$ fibers in line 4 and initializes $\mathbf{R}^{(t+1)}$, $\mathbf{U}^{(t+1)}$, and $\Delta \mathbf{R}$ with $\mathbf{R}^{(t)}$, $\mathbf{U}^{(t)}$, and an empty matrix respectively in line 5, where $\Delta \mathbf{R}$ consists of differences between $\mathbf{R}^{(t)}$ and $\mathbf{R}^{(t+1)}$.
In lines 6-16, CTD-D expands $\mathbf{R}^{(t+1)}$ with those sampled fibers by sequentially evaluating linear dependency of each fiber with the column space of $\mathbf{R}^{(t+1)}$. $\mathbf{R}^{(t+1)}$ and $\mathbf{U}^{(t+1)}$ are updated in this step. Finally, $\mathbf{C}_{(\alpha)}^{(t+1)}$ is updated in lines 17-21.

In the following, we describe two main ideas of CTD-D to update $\mathbf{C}_{(\alpha)}^{(t+1)}$, $\mathbf{R}^{(t+1)}$, and $\mathbf{U}^{(t+1)}$ efficiently while preserving accuracy: exploiting factors at previous time step, and reordering operations.

\noindent \textbf{(1) Exploiting factors at previous time step:}
First, we explain how we update $\mathbf{R}^{(t+1)}$ and $\mathbf{U}^{(t+1)}$ using the idea.
In line 5 of Algorithm \ref{alg:CTD-S}, CTD-S initializes $\mathbf{R}$ and $\mathbf{U}$ using one of the sampled fibers.
This is because CTD-S requires $\mathbf{R}$ to consist of linearly independent columns and it is satisfied when $\mathbf{R}$ has only one fiber.
Since $\mathbf{R}^{(t)}$ already consists of linearly independent columns, we initialize $\mathbf{R}^{(t+1)}$ and $\mathbf{U}^{(t+1)}$ with $\mathbf{R}^{(t)}$ and $\mathbf{U}^{(t)}$ respectively in line 5 of Algorithm \ref{alg:CTD-D}.
In lines 6-16, we check linear independence of each sampled fiber from $\Delta \mathbf{X}_{(\alpha)}$ with $\mathbf{R}^{(t+1)}$ .
If the fiber is linearly independent, we expand $\mathbf{R}^{(t+1)}$ and update $\mathbf{U}^{(t+1)}$ as in the lines 11-13 of Algorithm \ref{alg:CTD-S}.

Second, we describe how we update $\T{C}^{(t+1)}$ using the idea. We assume that $\Delta \mathbf{R}$ is not empty after line 16 of Algorithm \ref{alg:CTD-D}.
At time step $t$ and its successor step $t + 1$, CTD-S satisfies Equations \ref{eq:C_t} and \ref{eq:C_t+1}, where $\mathbf{C}_{(\alpha)}^{(t)}$ has the size $\tilde{d}_t \times N_\alpha^{(t)}$ and $\mathbf{C}_{(\alpha)}^{(t+1)}$ has the size $\tilde{d}_{t+1} \times N_\alpha^{(t+1)}$.
\begin{equation} \label{eq:C_t}
\mathbf{C}_{(\alpha)}^{(t)} \leftarrow (\mathbf{R}^{(t)})^T \mathbf{X}_{(\alpha)}^{(t)}
\end{equation}
\begin{equation} \label{eq:C_t+1}
\mathbf{C}_{(\alpha)}^{(t+1)} \leftarrow (\mathbf{R}^{(t+1)})^T \mathbf{X}_{(\alpha)}^{(t+1)}
\end{equation}

We can rewrite $\mathbf{R}^{(t+1)}$ and $\mathbf{X}_{(\alpha)}^{(t+1)}$ as Equations \ref{eq:R_t+1} and \ref{eq:X_t+1} respectively, where $\Delta \mathbf{R}$ has the size $I_\alpha \times \Delta \tilde{d}$ and $\Delta\mathbf{X}_{(\alpha)}$ has the size $I_\alpha \times \Delta N_\alpha$ such that $N_\alpha^{(t+1)} = N_\alpha^{(t)} + \Delta N_\alpha$ and $\tilde{d}_{t+1} = \tilde{d}_{t} + \Delta \tilde{d}$.
\begin{equation} \label{eq:R_t+1}
\mathbf{R}^{(t+1)} = \begin{bmatrix}
\mathbf{R}^{(t)} & \Delta \mathbf{R}
\end{bmatrix}
\end{equation}
\begin{equation} \label{eq:X_t+1}
\mathbf{X}_{(\alpha)}^{(t+1)} = \begin{bmatrix}
\mathbf{X}_{(\alpha)}^{(t)} & \Delta\mathbf{X}_{(\alpha)}
\end{bmatrix}
\end{equation}

We replace $\mathbf{R}^{(t+1)}$ and $\mathbf{X}_{(\alpha)}^{(t+1)}$ in Equation \ref{eq:C_t+1} with those in Equations \ref{eq:R_t+1} and \ref{eq:X_t+1}, respectively, to obtain the Equation \ref{eq:C_t+1_rewrite}.
\begin{equation} \label{eq:C_t+1_rewrite}
\begin{aligned}
\mathbf{C}_{(\alpha)}^{(t+1)}
& \leftarrow \begin{bmatrix}
(\mathbf{R}^{(t)})^T \\
(\Delta \mathbf{R})^T
\end{bmatrix}
\begin{bmatrix}
\mathbf{X}_{(\alpha)}^{(t)} & \Delta \mathbf{X}_{(\alpha)}
\end{bmatrix} \\
& = \begin{bmatrix}
(\mathbf{R}^{(t)})^T \mathbf{X}_{(\alpha)}^{(t)} & (\mathbf{R}^{(t)})^T \Delta \mathbf{X}_{(\alpha)} \\
(\Delta \mathbf{R})^T \mathbf{X}_{(\alpha)}^{(t)} & (\Delta \mathbf{R})^T \Delta \mathbf{X}_{(\alpha)}
\end{bmatrix}
\end{aligned}
\end{equation}

CTD-S computes all the 4 elements ($(\mathbf{R}^{(t)})^T \mathbf{X}_{(\alpha)}^{(t)}$, $(\mathbf{R}^{(t)})^T \Delta \mathbf{X}_{(\alpha)}$, $(\Delta \mathbf{R})^T \mathbf{X}_{(\alpha)}^{(t)}$, and $(\Delta \mathbf{R})^T \Delta \mathbf{X}_{(\alpha)}$) in Equation \ref{eq:C_t+1_rewrite} from scratch, hence requires a lot of computations. To make computation of $\mathbf{C}_{(\alpha)}^{(t+1)}$ incremental, we exploit existing factors at time step $t$ : $\mathbf{C}_{(\alpha)}^{(t)}$, $\mathbf{R}^{(t)}$, and $\mathbf{U}^{(t)}$. First, we use $\mathbf{C}_{(\alpha)}^{(t)}$ instead of $(\mathbf{R}^{(t)})^T \mathbf{X}_{(\alpha)}^{(t)}$ as in the Equation \ref{eq:C_t}. Second, we should replace $\mathbf{X}_{(\alpha)}^{(t)}$ in $(\Delta \mathbf{R})^T \mathbf{X}_{(\alpha)}^{(t)}$ with the factors at time step $t$, since CTD-D does not have $\mathbf{X}_{(\alpha)}^{(t)}$ as its input unlike CTD-S. We substitute $\mathbf{R}^{(t)} \mathbf{U}^{(t)} \mathbf{C}_{(\alpha)}^{(t)}$ for $\mathbf{X}_{(\alpha)}^{(t)}$. This is because CTD-S ensures $\T{X}^{(t)} \approx \T{C}^{(t)}\times_\alpha \mathbf{R}^{(t)}\mathbf{U}^{(t)}$ which can be rewritten as $\mathbf{X}_{(\alpha)}^{(t)} \approx \mathbf{R}^{(t)} \mathbf{U}^{(t)} \mathbf{C}_{(\alpha)}^{(t)}$ by Equation \ref{eq:n_mode_product_property}. Equation \ref{eq:C_t+1_final} shows the final form of $\mathbf{C}_{(\alpha)}^{(t+1)}$ which is the same as line 18 in Algorithm \ref{alg:CTD-D}.
\begin{equation} \label{eq:C_t+1_final}
\mathbf{C}_{(\alpha)}^{(t+1)} \leftarrow \begin{bmatrix}
\mathbf{C}_{(\alpha)}^{(t)} & (\mathbf{R}^{(t)})^T \Delta \mathbf{X}_{(\alpha)} \\
(\Delta \mathbf{R})^T \mathbf{R}^{(t)} \mathbf{U}^{(t)} \mathbf{C}_{(\alpha)}^{(t)} & (\Delta \mathbf{R})^T \Delta \mathbf{X}_{(\alpha)}
\end{bmatrix}
\end{equation}
$(\Delta \mathbf{R})^T \mathbf{R}^{(t)} \mathbf{U}^{(t)} \mathbf{C}_{(\alpha)}^{(t)}$ and $(\Delta \mathbf{R})^T \Delta \mathbf{X}_{(\alpha)}$ are ignored when $\Delta \mathbf{R}$ is empty as expressed in line 20 of Algorithm \ref{alg:CTD-D}.

\noindent \textbf{(2) Reordering computations:}
The computation order for the element $(\Delta \mathbf{R})^T \mathbf{R}^{(t)} \mathbf{U}^{(t)} \mathbf{C}_{(\alpha)}^{(t)}$ is important since each order has a different computation cost. We want to determine the optimal parenthesization among possible parenthesizations.
It can be shown that
$(((\Delta \mathbf{R})^T \mathbf{R}^{(t)}) \mathbf{U}^{(t)}) \mathbf{C}_{(\alpha)}^{(t)}$ is the optimal one with $\mathcal{O}((\Delta \tilde{d}) \tilde{d}_t (I_\alpha + \tilde{d}_t + N_\alpha^{(t)}))$ operations and can be done by parenthesizing from the left.

We prove that CTD-D is faster than CTD-S in Lemma \ref{lemma: CTD-D complexity}.
\begin{lemma} \label{lemma: CTD-D complexity}
	CTD-D is faster than CTD-S. The computational complexity of CTD-D is $\mathcal{O}((\Delta \tilde{d}) \tilde{d}_t (N_\alpha^{(t)} + I_\alpha) + (\tilde{d}_{t+1} I_\alpha + d) \Delta N_\alpha + d' (\tilde{d}_{t+1}^2 + nnz(\mathbf{R}^{(t+1)})) + d \log d + nnz(\Delta \T{X}))$.
\end{lemma}

\begin{proof}	
	The lines 1-4 of Algorithm \ref{alg:CTD-D} for CTD-D are similar to those of Algorithm \ref{alg:CTD-S} for CTD-S. The only difference is that CTD-D samples $d$ columns from $\Delta \mathbf{X}_{(\alpha)}$ while CTD-S samples $s$ columns from $\mathbf{X}_{(\alpha)}$. Thus, lines 1-4 takes $\mathcal{O}(nnz(\Delta \T{X}) + d \Delta N_\alpha + d \log d)$. Updating $\mathbf{R}^{(t+1)}$ and $\mathbf{U}^{(t+1)}$ in lines 5-16 needs $\mathcal{O}(d' (\tilde{d}_{t+1}^2 + nnz(\mathbf{R}^{(t+1)})))$ operations as proved in Lemma 1 in \cite{tong2008colibri}. In updating $\T{C}^{(t+1)}$ in lines 17-18, $(\mathbf{R}^{(t)})^T \Delta \mathbf{X}_{(\alpha)}$ takes computational cost of $\mathcal{O}(\tilde{d}_t I_\alpha \Delta N_\alpha)$. $(\Delta \mathbf{R})^T \Delta \mathbf{X}_{(\alpha)}$ takes $\mathcal{O}(\Delta \tilde{d} I_\alpha \Delta N_\alpha)$ and $(\Delta \mathbf{R})^T \mathbf{R}^{(t)} \mathbf{U}^{(t)} \mathbf{C}_{(\alpha)}^{(t)}$ takes $\mathcal{O}((\Delta \tilde{d}) \tilde{d}_t (I_\alpha + \tilde{d}_t + N_\alpha^{(t)}))$.
Overall, CTD-D takes $\mathcal{O}((\Delta \tilde{d}) \tilde{d}_t (N_\alpha^{(t)} + I_\alpha) + (\tilde{d}_{t+1} I_\alpha + d) \Delta N_\alpha + d' (\tilde{d}_{t+1}^2 + nnz(\mathbf{R}^{(t+1)})) + d \log d + nnz(\Delta \T{X}))$.
	
	CTD-D is faster than CTD-S because CTD-S has $\tilde{s} I_\alpha N_\alpha$ in its complexity, which is much larger than all the terms in the complexity of CTD-D.\hfill
\end{proof}

\section{Experiments}
\label{sec:experiments}
In this and the next sections, we perform experiments to answer the following questions. \\

\begin{table}[t]
	\small
	\setlength{\tabcolsep}{3pt}
	\caption{Summary of the tensor data used.}
	\begin{center}
		{	
			\begin{tabular}{C{3cm} R{0.7cm} R{0.7cm} R{0.7cm} R{2cm}}
				\toprule
				
				\textbf{Name} & $I_1$ & $I_2$ & $I_3$ & \textbf{Nonzeros} \\
				
				\midrule
				
				Hypertext 2009\tablefootnote{http://konect.uni-koblenz.de/networks/sociopatterns-hypertext} & 112 & 113 & 5,246 & 20,818 \\
				
				Haggle\tablefootnote{http://konect.uni-koblenz.de/networks/contact}& 77 & 274 & 1,567 & 27,972 \\
				
				Manufacturing emails\tablefootnote{http://konect.uni-koblenz.de/networks/radoslaw\_email} & 167 & 166 & 2,615 & 70,523 \\
				
				 Infectious\tablefootnote{http://konect.uni-koblenz.de/networks/sociopatterns-infectious} & 407 & 410 & 1,392 & 17,298 \\
				
				
				CAIDA (synthetic) & 189 & 189 & 1,000 & 20,511 \\
				
				
				\bottomrule
				
			\end{tabular}
		}
	\end{center}
	\label{tab:dataset}
\end{table}

\noindent \textbf{Q1} : What is the performance of our static method CTD-S compared to the competing method tensor-CUR? (Section \ref{subsec:Performance of CTD-S})\\
\noindent \textbf{Q2} : How do the performance of CTD-S and tensor-CUR change with regard to the sample size parameter? (Section \ref{subsec:Performance of CTD-S})\\
\noindent \textbf{Q3} : What is the performance of our dynamic method CTD-D compared to the static method CTD-S? (Section \ref{subsec:Performance of CTD-D})\\
\noindent \textbf{Q4} : What is the result of applying CTD-D for online DDoS attack detection? (Section \ref{sec:ctd at works})\\
\vspace{-0.3cm}
\subsection{Experimental Settings}
\paragraph{Measure.} We define three metrics (1. relative error, 2. memory usage, and 3. running time) as follows.
First, a relative error is defined as Equation \ref{eq:error}. $\T{X}$ denotes the original tensor and $\tilde{\T{X}}$ is the tensor reconstructed from the factors of $\T{X}$. For example, $\tilde{\T{X}} = \T{C}\times_\alpha \mathbf{R}\mathbf{U}$ in CTD-S.
\begin{equation}\label{eq:error}
Relative Error = \frac{||\tilde{\T{X}} - \T{X}||_F^2}{||\T{X}||_F^2}
\end{equation}
Second, memory usage is defined as Equation \ref{eq:memoryEff}.
Memory usage measures the relative amount of memory for storing the resulting factors. The denominator and numerator indicate the amount of memory for storing the original tensor and the resulting factors, respectively.

\begin{equation}\label{eq:memoryEff}
Memory Usage = \frac{nnz(\T{C}) + nnz(\mathbf{U}) + nnz(\mathbf{R})}{nnz(\T{X})}
\end{equation}
Finally, running time is measured in seconds.

\paragraph{Data.} Table \ref{tab:dataset} shows the data we used in our experiments. 

\paragraph{Machine.} All the experiments are performed on a machine with a 10-core Intel 2.20 GHz CPU and 256 GB RAM.

\paragraph{Competing method.} We compare our proposed method CTD with tensor-CUR \cite{mahoney2008tensor}, the state-of-the-art sampling-based tensor decomposition method. Both methods are implemented in MATLAB.

\begin{figure*} [t]
	\subfloat[\textbf{Relative error vs. sample size}]
	{	\includegraphics[width=0.30 \textwidth]{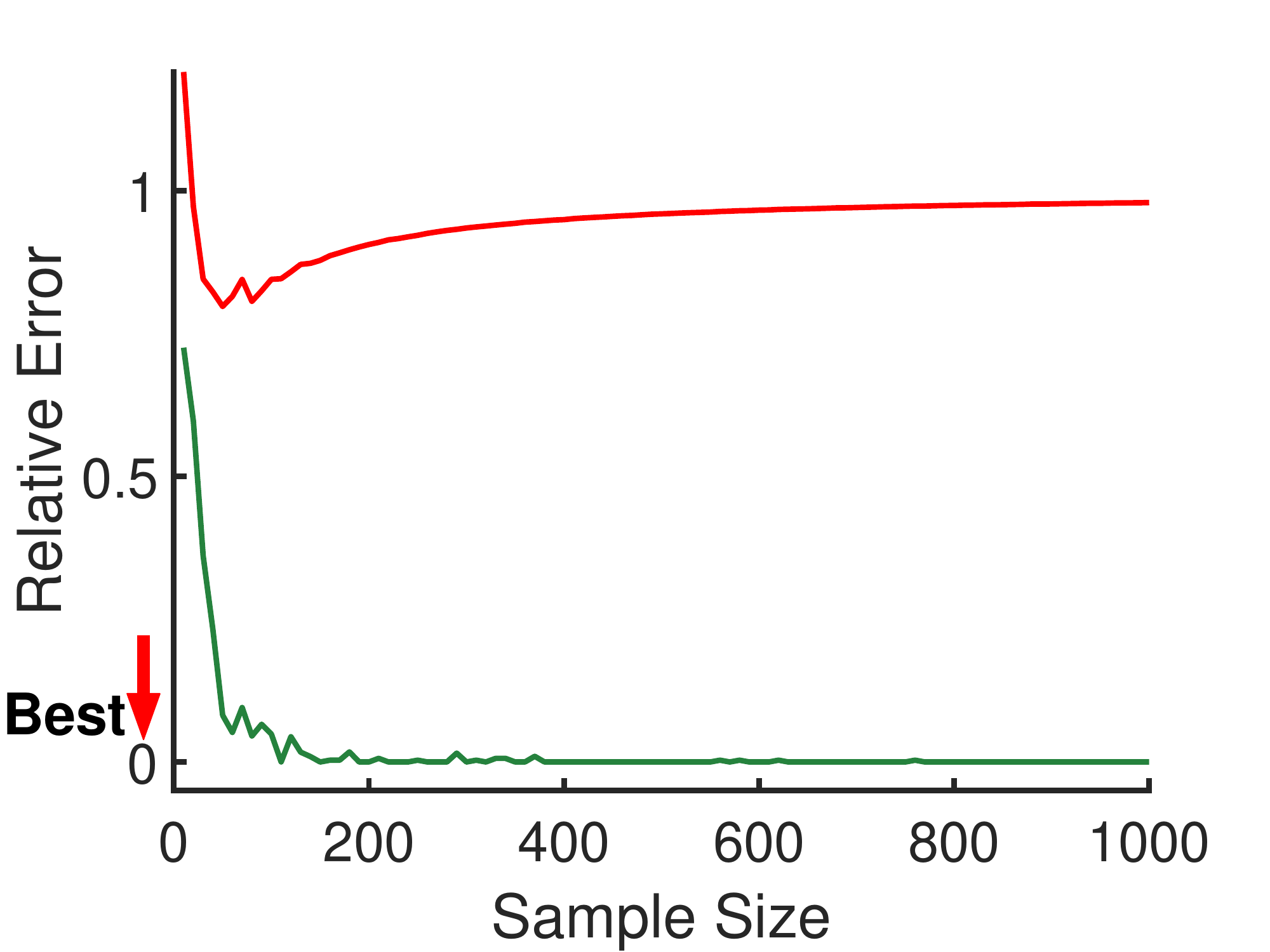}
	}
	\subfloat[\textbf{Running time vs. sample size}]
	{	\includegraphics[width=0.30 \textwidth]{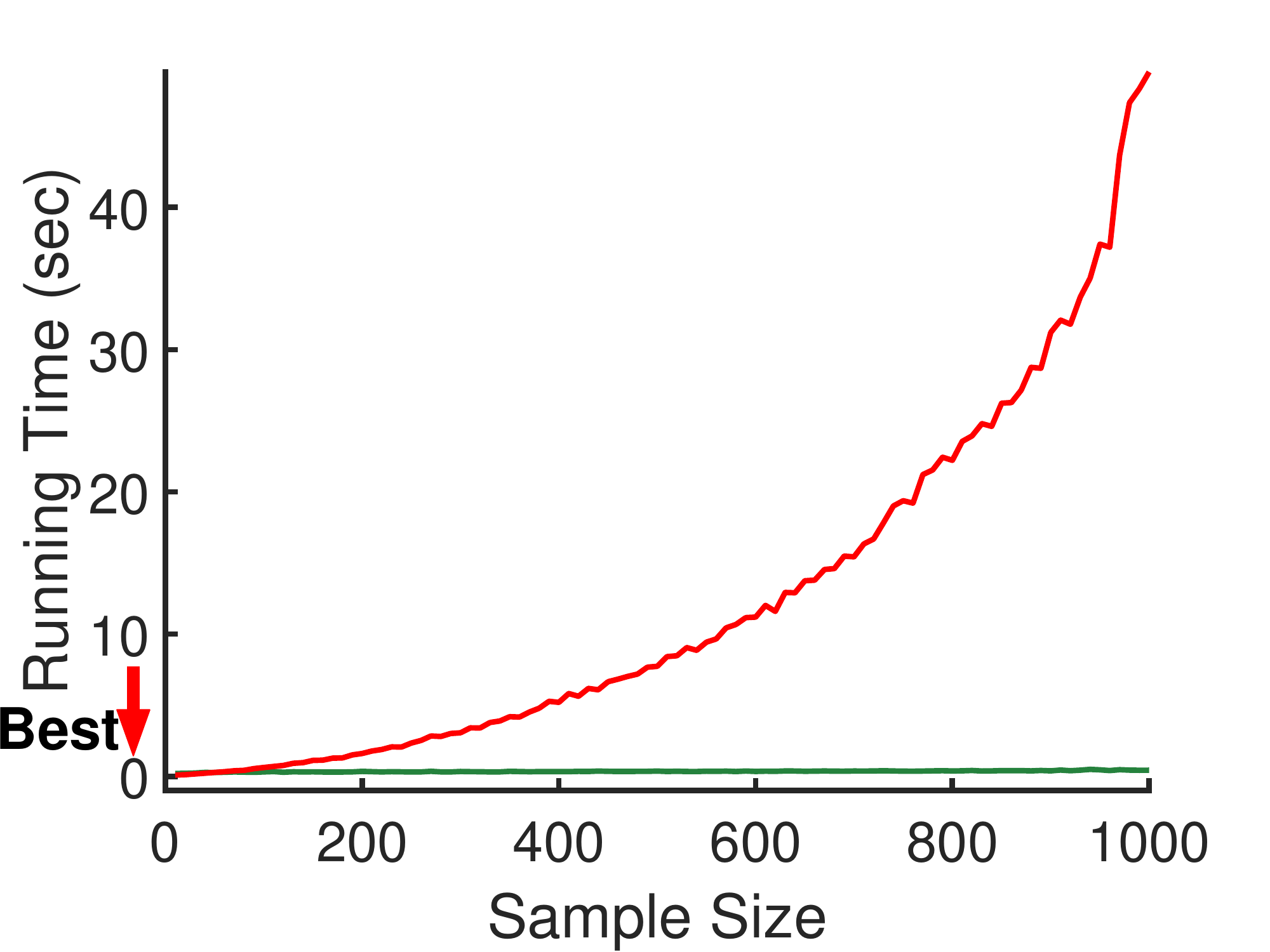}
	}
	\subfloat[\textbf{Memory usage vs. sample size}]
	{	\includegraphics[width=0.30 \textwidth]{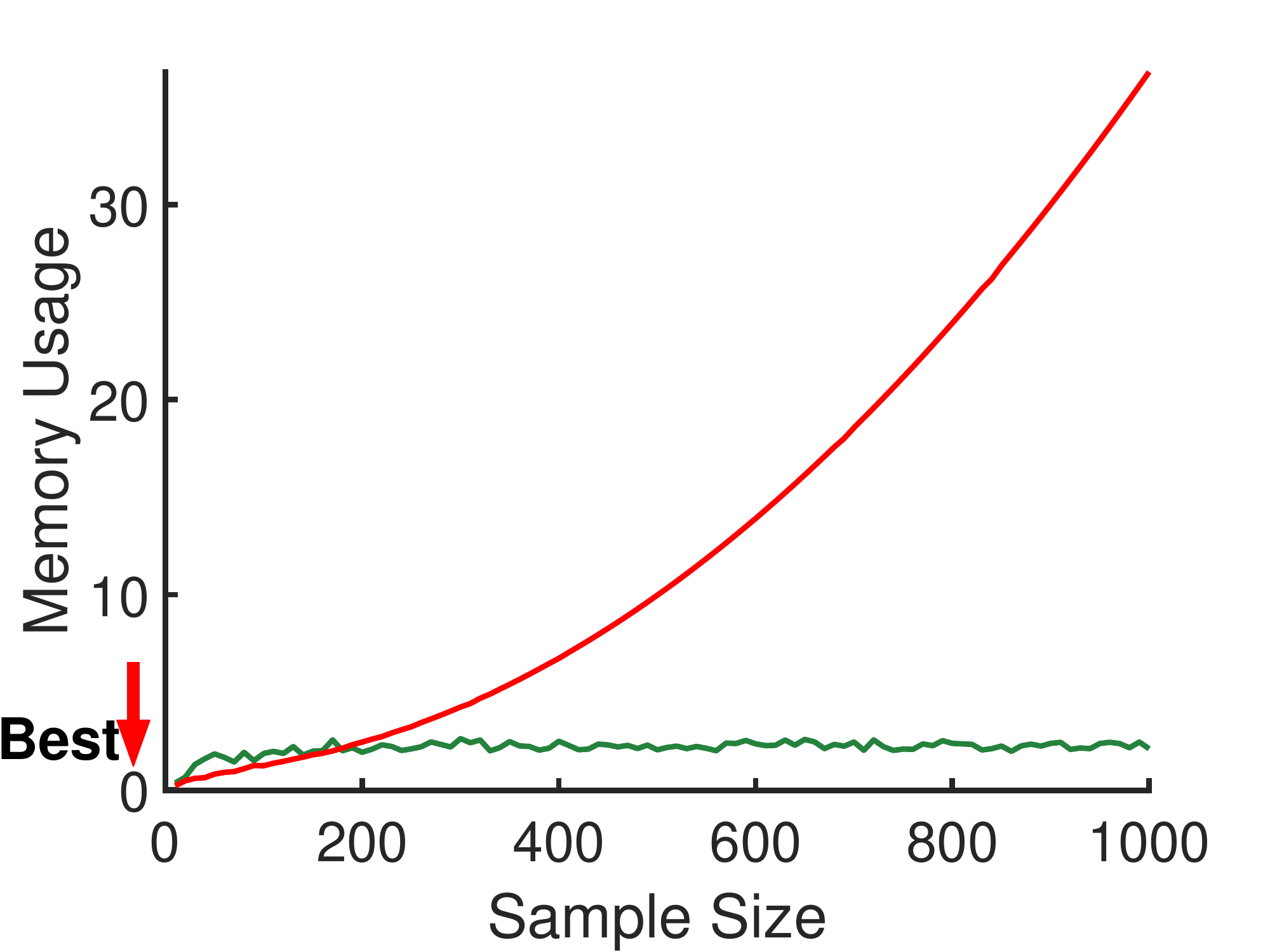}
	}
	
	\includegraphics[width=0.35 \textwidth]{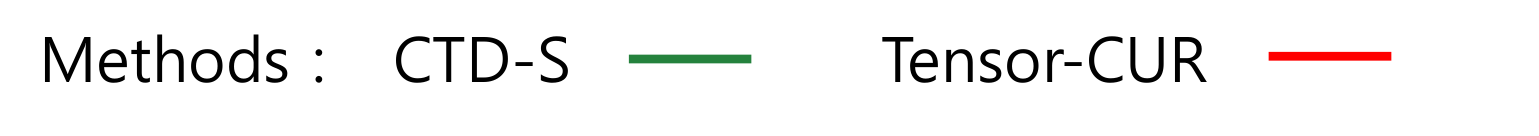}
	
	\caption{Error, running time, and memory usage of CTD-S compared to those of tensor-CUR over sample size $s$ for Haggle dataset. CTD-S is more accurate over various sample sizes, and its running time and memory usage are relatively constant compared to the tensor-CUR.}
	\label{fig:CTD-S by c}
\end{figure*}

\subsection{Performance of CTD-S} \label{subsec:Performance of CTD-S}
We measure the performance of CTD-S to answer Q1 and Q2.
As a result, CTD-S is 17$\sim$83$\times$ more accurate for the same level of running time,
5$\sim$86$\times$ faster, and 7$\sim$12$\times$ more memory-efficient for the same level of error compared to tensor-CUR.
CTD-S is more accurate over various sample sizes and its running time and memory usage are relatively constant compared to the tensor-CUR.
The detail of the experiment is as follows.

Both CTD-S and tensor-CUR takes a given tensor $\T{X}$, mode $\alpha$, and a sample size $s$ as input because they are LR tensor decomposition methods. In each experiment, we give the same input and compare the performance. We set $\alpha = 1$ and perform experiments under various sample sizes $s$. We set the number of slabs to sample $r=s$ and the rank $k = 10$ in tensor-CUR, and set $\epsilon = 10^{-6}$ in CTD-S.

Figure \ref{fig:CTD-S all} shows the running time vs. error and memory usage vs. error of CTD-S compared to tensor-CUR, which are the answers for Q1.
We use sample sizes from 1 to 1000. The error of tensor-CUR is much larger than that of CTD-S. This phenomenon coincides with the Lemma \ref{lemma: CTD-S accuracy}, which guarantees that CTD-S is more accurate than tensor-CUR theoretically.
We compare running time and memory usage under the same level of error, not under the same sample size, because there is a huge gap between the error of CTD-S and that of tensor-CUR under the same sample size.
%

Figure \ref{fig:CTD-S by c} shows the error, running time, and memory usage of CTD-S compared to those of tensor-CUR over increasing sample sizes $s$ for the Haggle dataset, which are the answers for Q2. The error of CTD-S decreases as $s$ increases because it gets more chance to sample important fibers which describe the original tensor well. The running time and memory usage of CTD-S are relatively constant compared to those of tensor-CUR. This is because CTD-S keeps only the linearly independent fibers, the number of which is bound to the rank of $\mathbf{X}_{(\alpha)}$. There are small fluctuations in the graphs since the sampling process of both CTD-S and tensor-CUR are based on randomness.

\subsection{Performance of CTD-D} \label{subsec:Performance of CTD-D}
We compare the performance of CTD-D with those of CTD-S to answer Q3. As a result, CTD-D is 2$\sim$3$\times$ faster for the same level of error compared to CTD-S.

To simulate a dynamic environment, we divide a given dataset into two parts along the time mode. We use the first 80\% of the dataset as historical data and the later 20\% as incoming data. We assume that historical data is already given and incoming data arrives sequentially at every time step, such that the whole data grows along the time mode. We measure the performance of CTD-D and CTD-S at each time step and calculate the average.
We set the sample size $d$ of CTD-D to be much smaller than that of CTD-S because CTD-D samples fibers only from the increment $\Delta \T{X}$ while CTD-S samples from the whole data $\T{X}$.
We set $d$ of CTD-D to be 0.01 times of $s$ of CTD-S, $\alpha = 1$, and $\epsilon = 10^{-6}$.

\begin{figure*} [t]
	\subfloat[\textbf{Hypertext 2009}]
	{	\includegraphics[width=0.23 \textwidth]{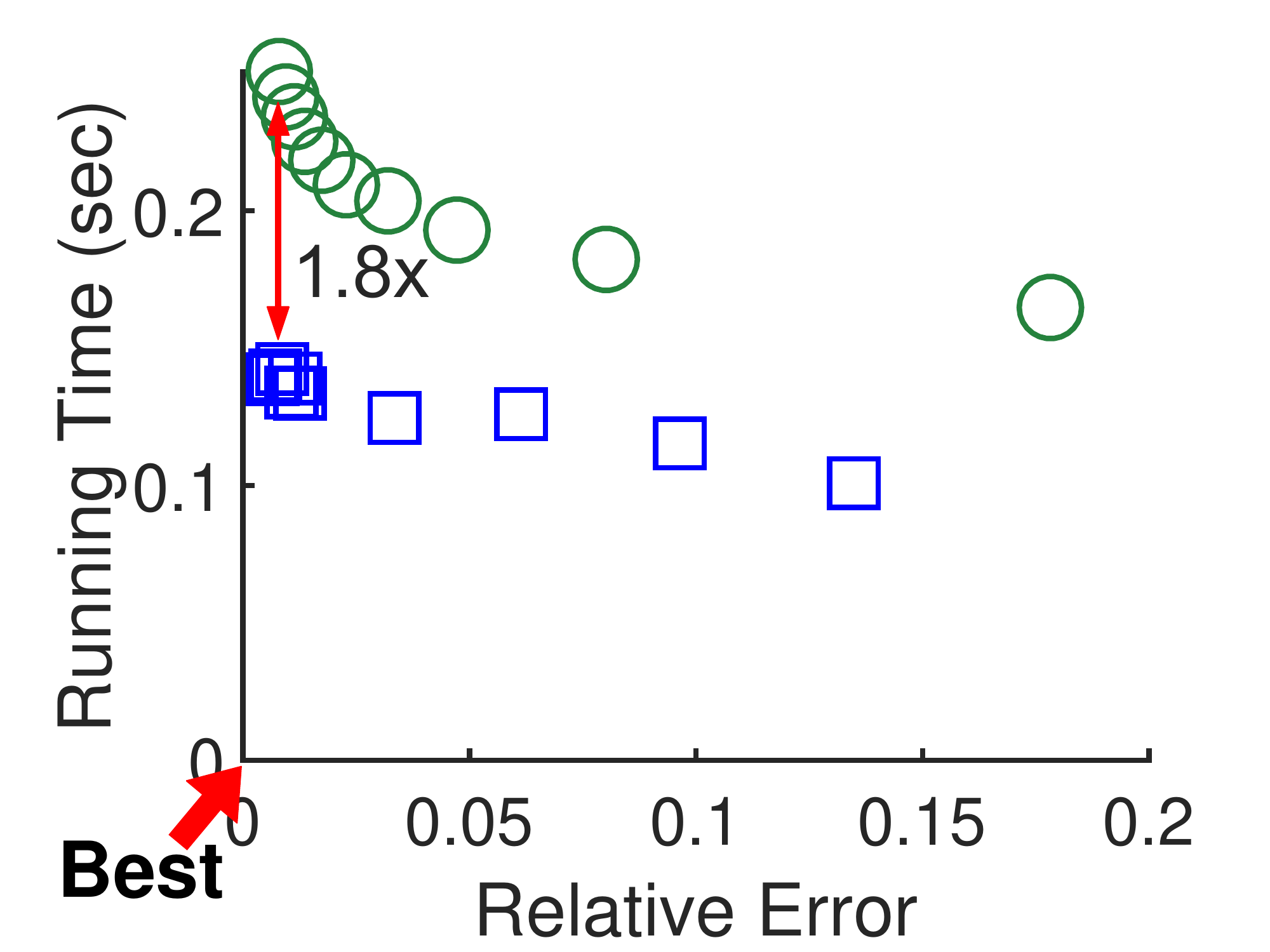}
	}
	\subfloat[\textbf{Haggle}]
	{	\includegraphics[width=0.23 \textwidth]{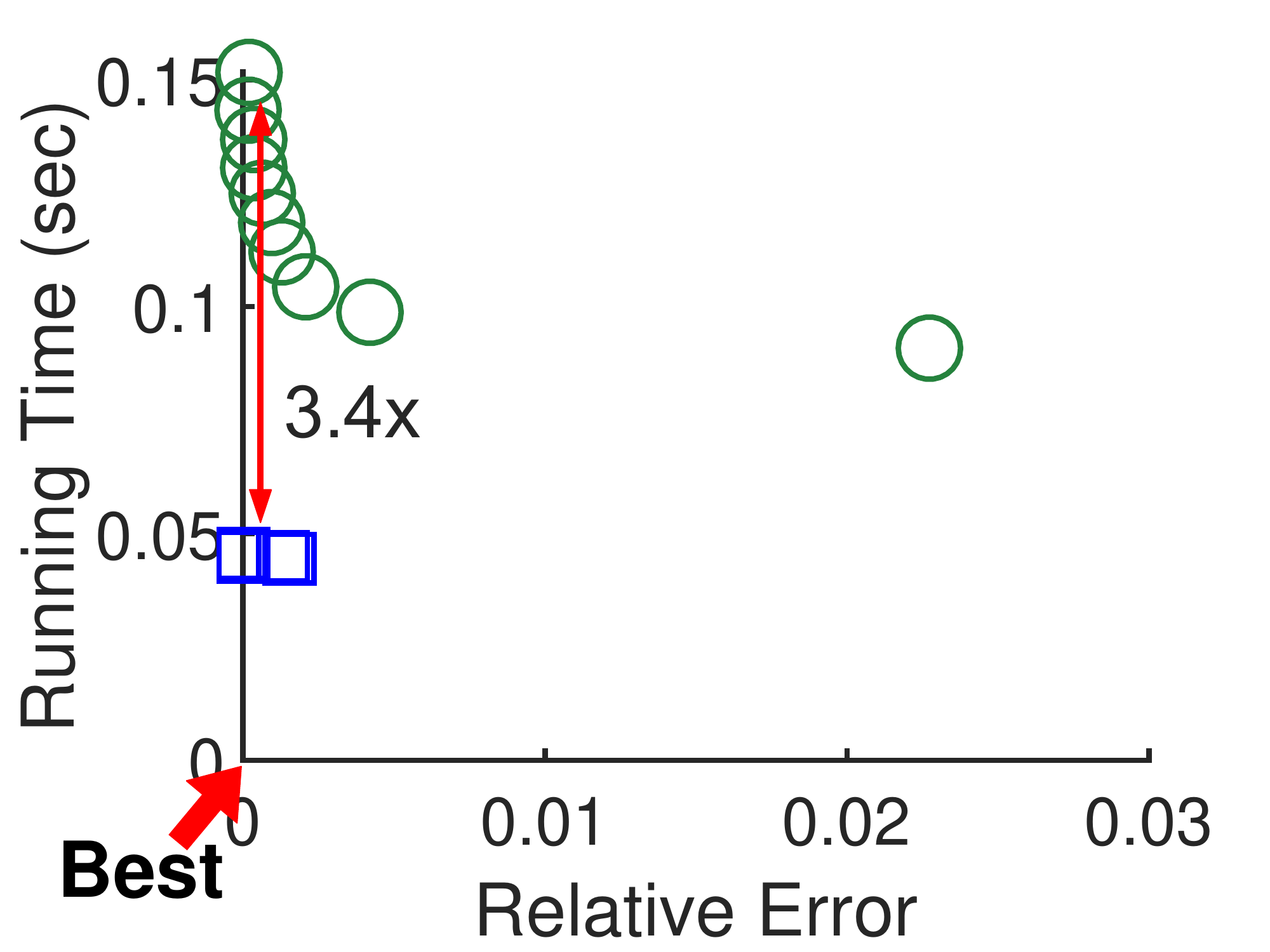}
	}
	\subfloat[\textbf{Manufacturing emails}]
	{	\includegraphics[width=0.23 \textwidth]{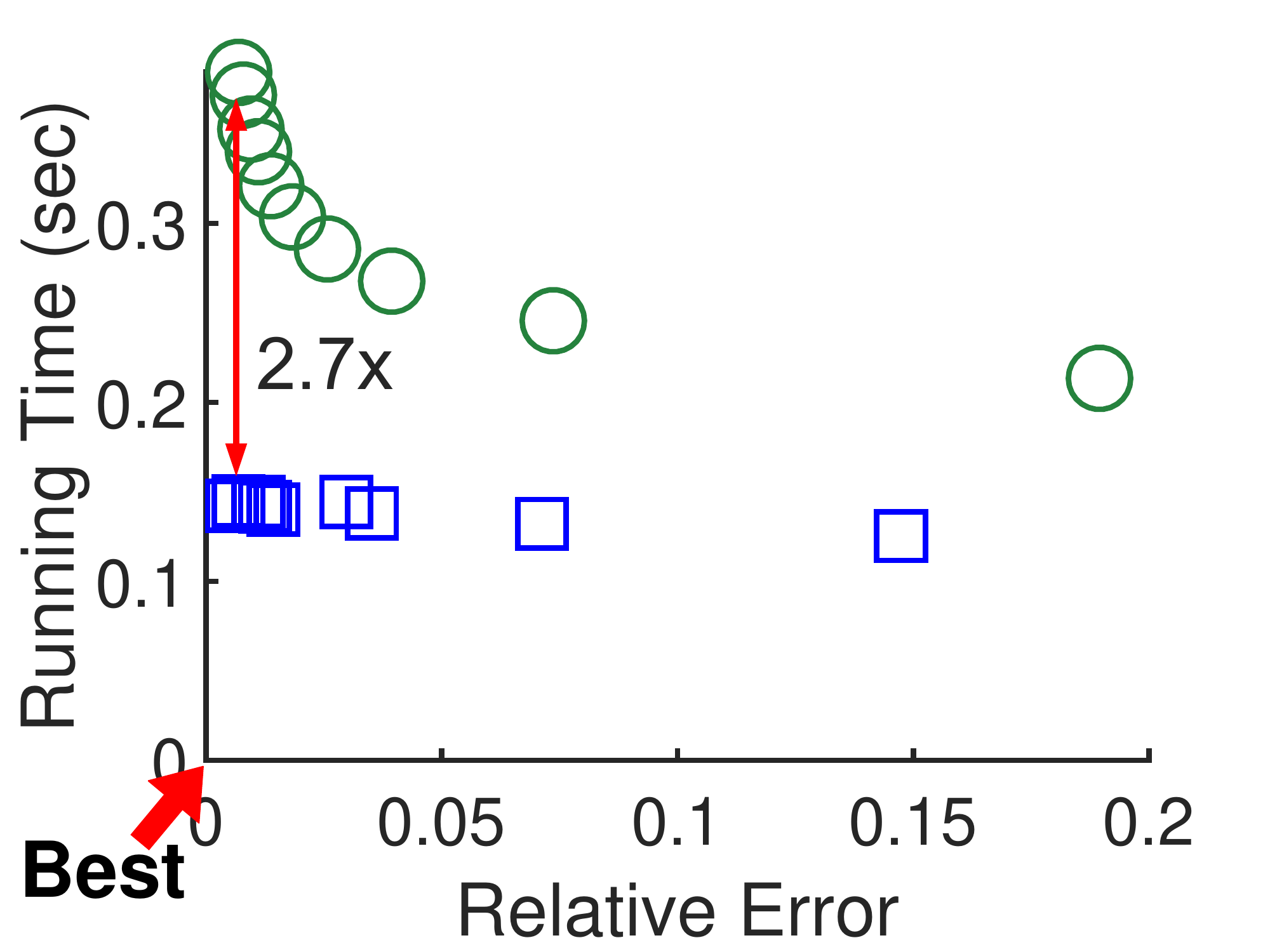}
	}
	\subfloat[\textbf{Infectious}]
	{	\includegraphics[width=0.23 \textwidth]{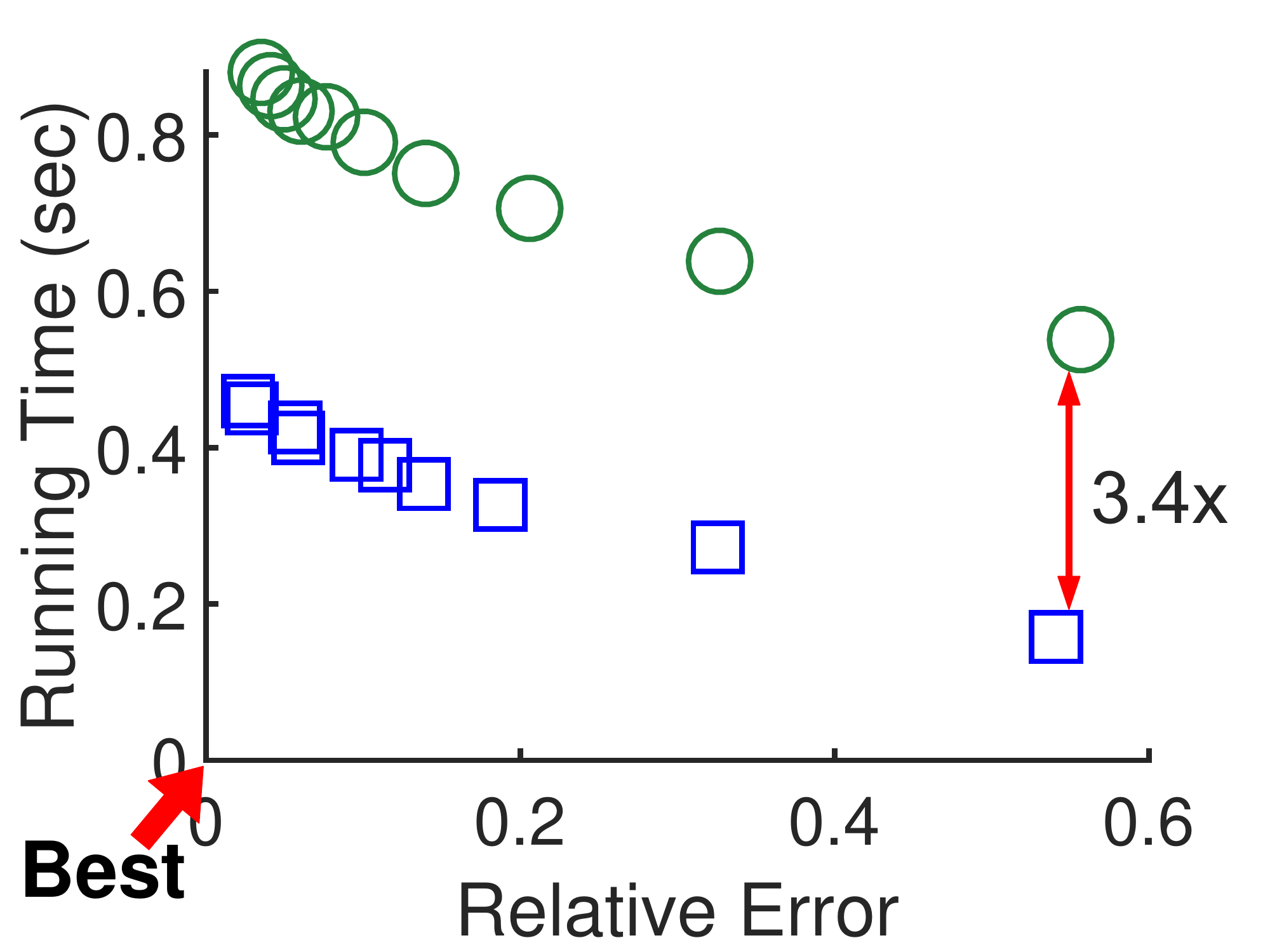}
	} \vspace*{-0.4cm}\\
	\subfloat[\textbf{Hypertext 2009}]
	{	\includegraphics[width=0.23 \textwidth]{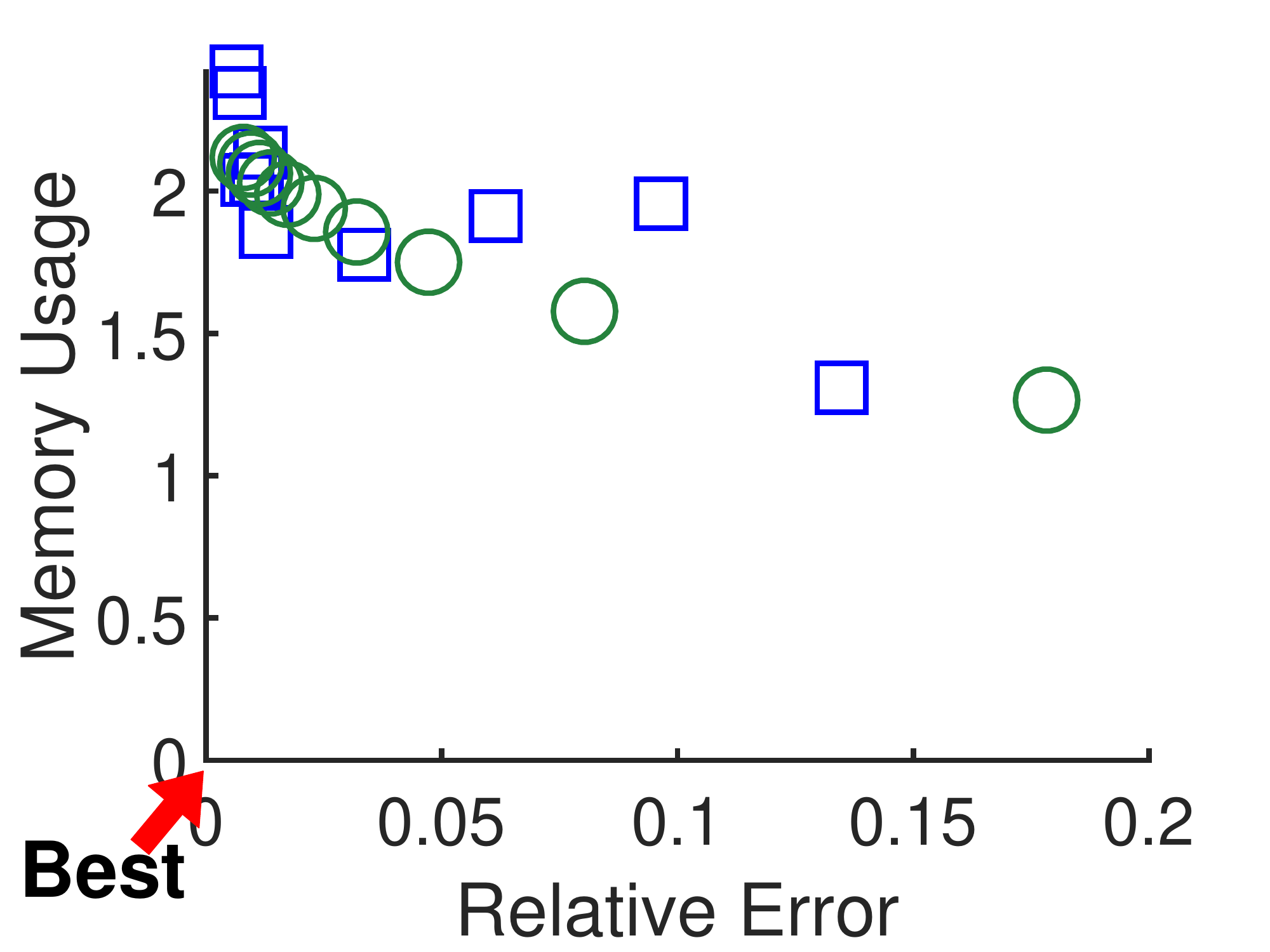}
	}
	\subfloat[\textbf{Haggle}]
	{	\includegraphics[width=0.23 \textwidth]{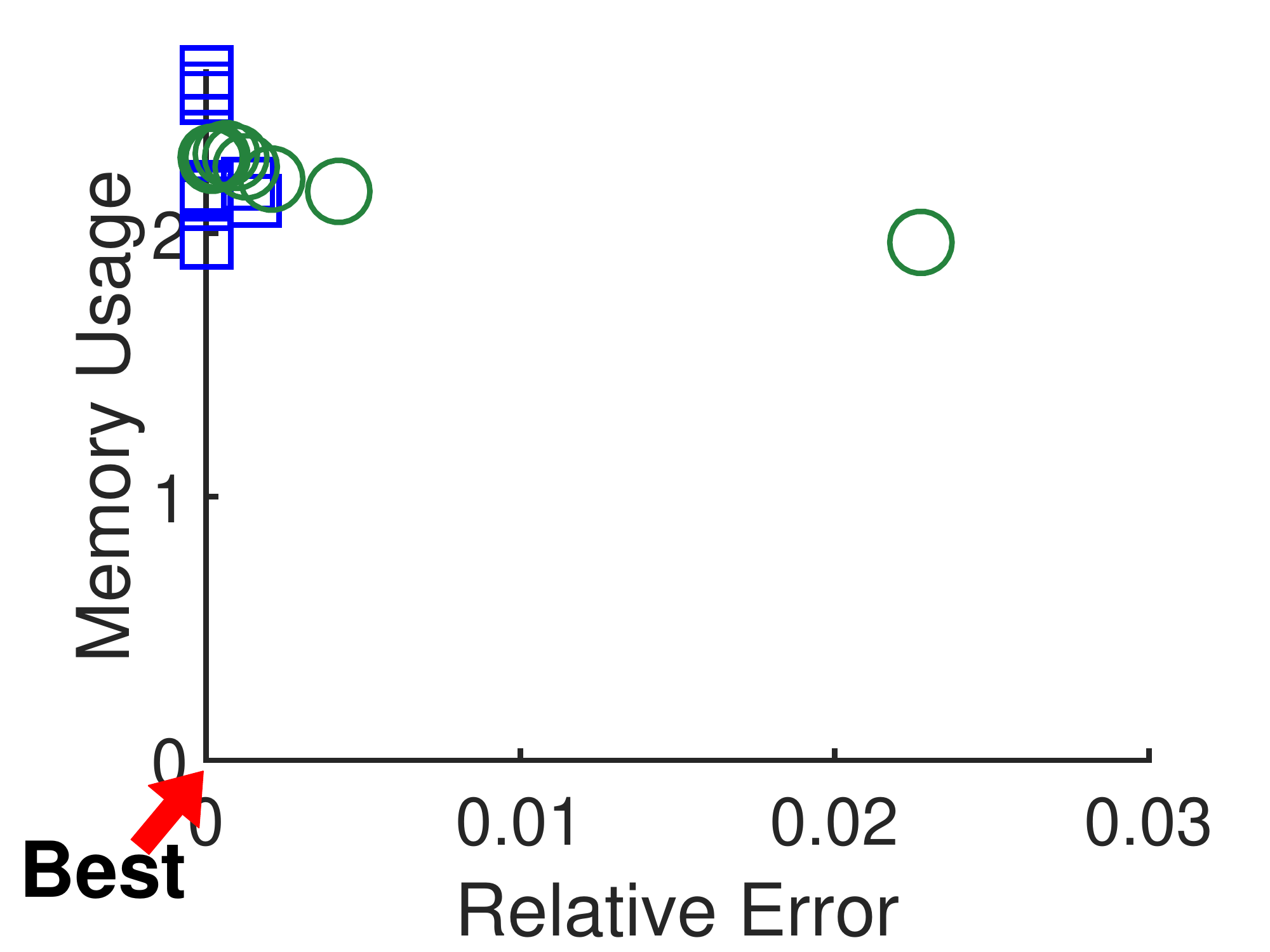}
	}
	\subfloat[\textbf{Manufacturing emails}]
	{	\includegraphics[width=0.23 \textwidth]{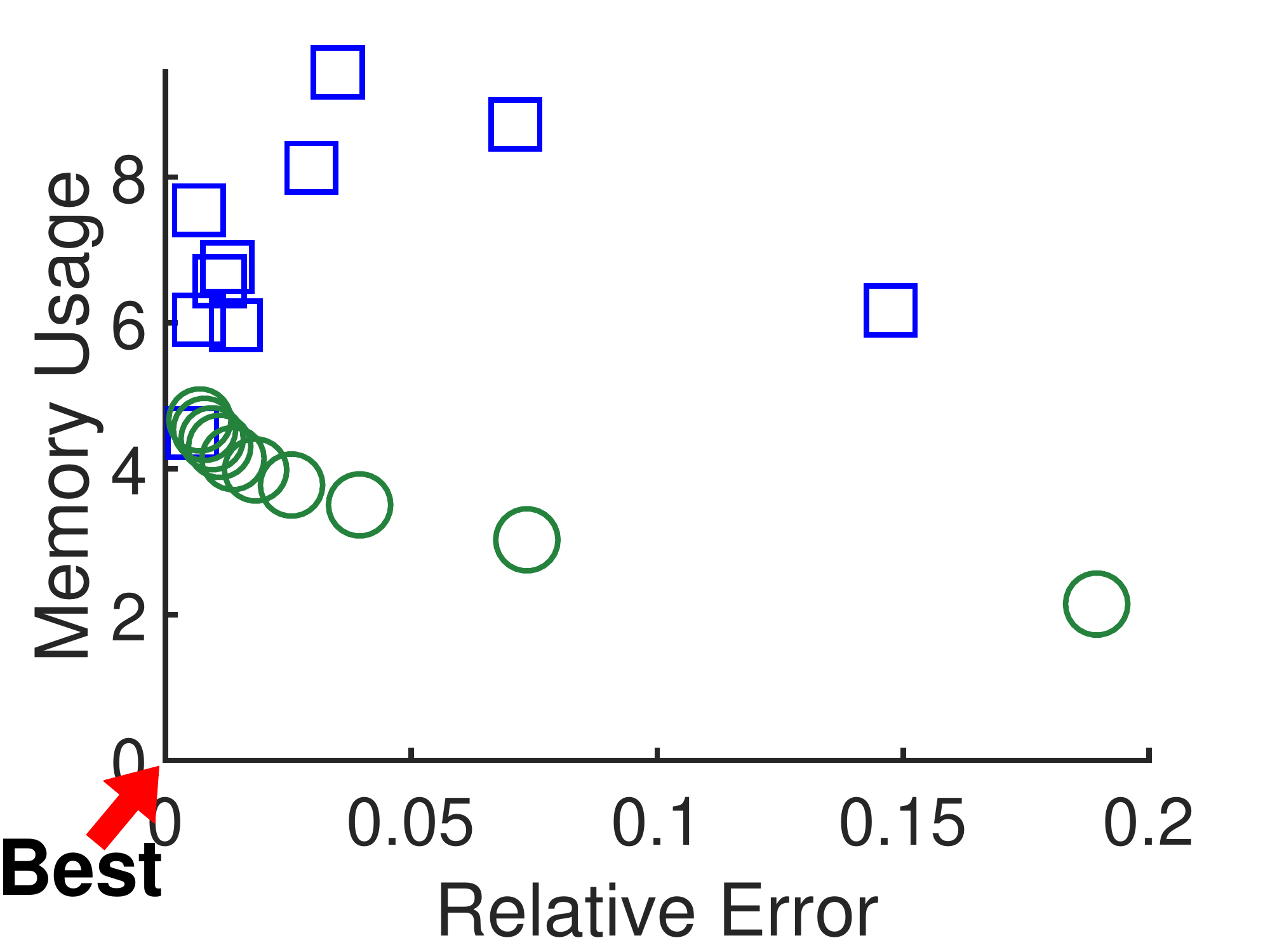}
	}
	\subfloat[\textbf{Infectious}]
	{	\includegraphics[width=0.23 \textwidth]{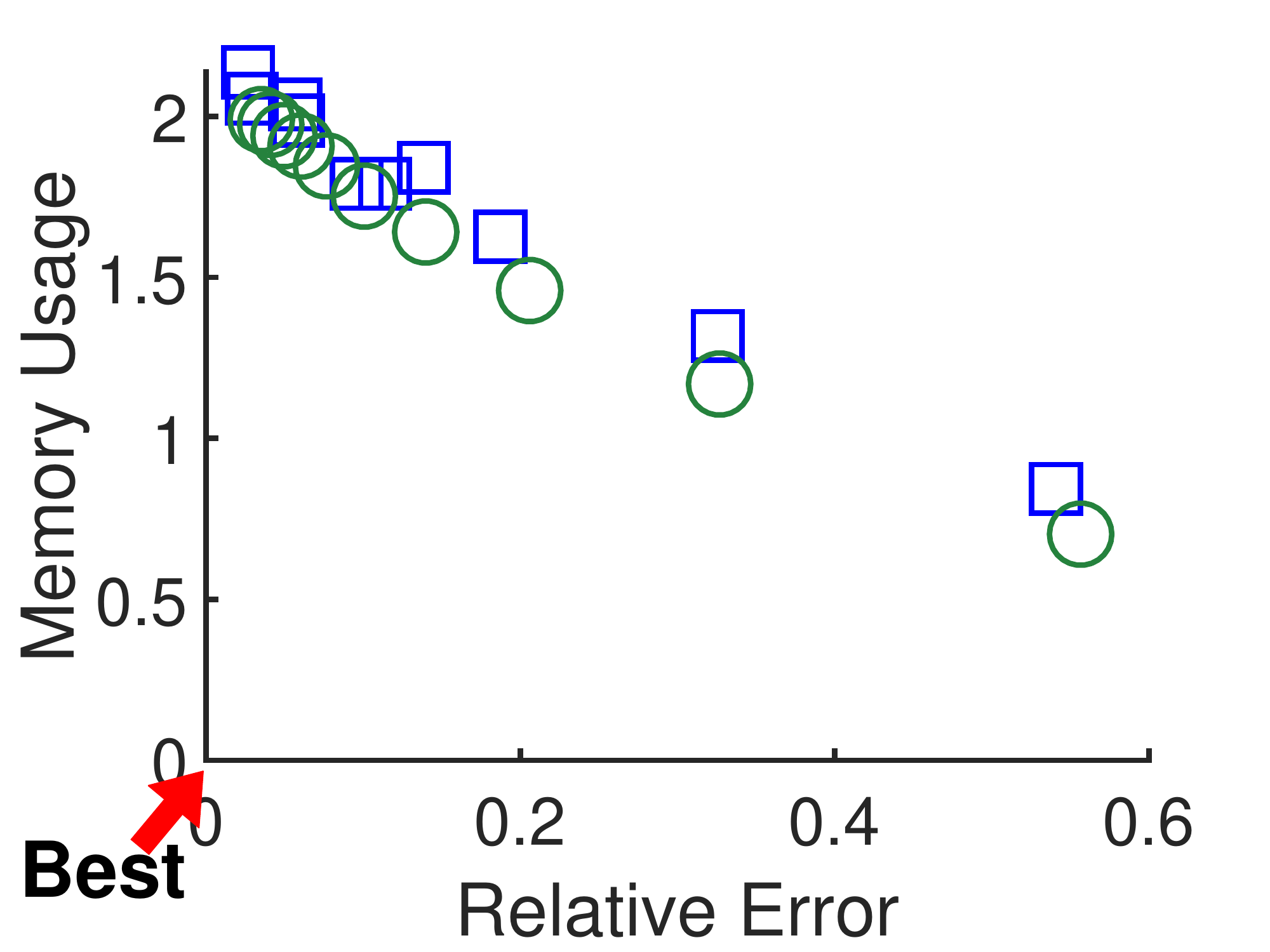}
	}
	
	\includegraphics[width=0.35 \textwidth]{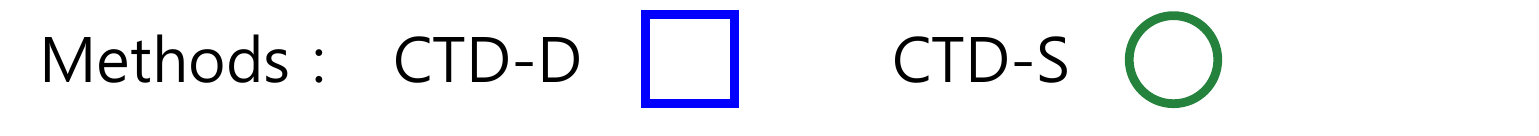}
	
	\caption{Error, running time, and memory usage relation of CTD-D compared to those of CTD-S. CTD-D is faster and has smaller error while using the same or slightly larger memory space compared to CTD-S.}
	\label{fig:CTD-D all}
\end{figure*}

Figure \ref{fig:CTD-D all} shows the error vs. running time and error vs. memory usage relation of CTD-D compared to those of CTD-S.
Note that CTD-D is much faster than CTD-S though CTD-D uses the same or slightly more memory than CTD-S does. This is because multiplication between sparse matrices used in updating $\T{C}$ does not always produce sparse output, thus the number of nonzero entries in $\T{C}$ increases slightly over time steps.

\section{CTD At Work}
\label{sec:ctd at works}
In this section, we use CTD-D for mining real-world network traffic tensor data.
Our goal is to detect DDoS attacks in network traffic data efficiently in an online fashion;
detecting DDoS attacks is a crucial task in network forensic.
We propose a novel online DDoS attack detection method based on CTD-D's interpretability.
We show that CTD-D is one of the feasible options for online DDoS attack detection and show how it detects DDoS attacks successfully. 
In contrast to the standard PARAFAC and Tucker decomposition methods, CTD-D can determine DDoS attacks from its decomposition result without expensive overhead. 
We aim to dynamically find a victim (destination host) and corresponding attackers (source hosts) of each DDOS attack in network traffic data. The victim receives a huge amount of traffic from a large number of attackers.

The online DDoS attack detection method based by CTD-D is as follows. First, we apply CTD-D on network traffic data which is a 3-mode tensor in the form of (source IP - destination IP - time).
We assume an online environment where each slab of the network traffic data in the form of (source IP - destination IP) arrives sequentially at every time step.
We use source IP mode as mode $\alpha$.
Second, we inspect the factor $\mathbf{R}$ of CTD-D, which consists of actual mode-$\alpha$ fibers from the original data.
$\mathbf{R}$ is composed of important mode-$\alpha$ fibers which signify major activities such as DDoS attack or heavy traffic to the main server.
Thanks to CTD, we directly find out destination host and occurrence time of a major activity represented in a fiber in $\mathbf{R}$, by simply tracking the indices of fibers.
We regard fibers with the same destination host index represent the same major activity, and consider the first fiber among those with the same destination host index to be the representative of each major activity.
Then, we select fibers with the norm higher than the average among the first fibers and suggest them as candidates of DDoS attack.
This is because DDoS attacks have much higher norms than normal traffic does.

We generate network traffic data by injecting DDoS attacks on the real-world CAIDA network traffic dataset. We assume that randomly selected 20\% of source hosts participate in each DDoS attack.
Table \ref{tab:network_result} shows the result of DDoS attack detection method of CTD-D. CTD-D achieves high F1 score for various number $n$ of injected DDoS attacks with notable precision. We set $d = 10$, $\epsilon = 0.15$.

\begin{table}[tp]
	\small
	\caption{The result of online DDoS attack detection method based on CTD-D. CTD-D achieves high F1 score for various $n$ with notable precision, where $n$ denotes the number of injected DDoS attacks.}
	\begin{center}
		{
			\begin{tabular}{C{1cm} C{2cm} C{2cm} C{2cm}}           \hline
				\toprule
				{\textbf{n}} &{\textbf{Recall}} & {\textbf{Precision}} & {\textbf{F1 score}} \\
				
				\midrule
				1 & 1.000 & 1.000 & 1.000 \\
				3 & 1.000 & 1.000 & 1.000 \\
				5 & 0.880 & 1.000 & 0.931 \\
				7 & 0.857 & 1.000 & 0.921 \\
				\bottomrule
			\end{tabular}
		}
	\end{center}
	\label{tab:network_result}
\end{table}


\section{Conclusion}
\label{sec:conclusions}
We propose CTD, a fast, accurate, and directly interpretable tensor decomposition method based on sampling.
The static version CTD-S is 17$\sim$83$\times$ more accurate, 5$\sim$86$\times$ faster, and 7$\sim$12$\times$ more memory-efficient compared to tensor-CUR, the state-of-the-art method.
The dynamic version CTD-D is 2$\sim$3$\times$ faster than CTD-S for an online environment. CTD-D is the first method providing interpretable \emph{dynamic} tensor decomposition.
We show the effectiveness of CTD for online DDoS attack detection.


\bibliographystyle{ACM-Reference-Format}
\bibliography{BIB/ms}

\end{document}